\documentclass[sn-mathphys,Numbered]{sn-jnl}

\usepackage{graphicx}%
\usepackage{multirow}%
\usepackage{amsmath,amssymb,amsfonts}%
\usepackage{amsthm}%
\usepackage{mathrsfs}%
\usepackage[title]{appendix}%
\usepackage{xcolor}%
\usepackage{textcomp}%
\usepackage{manyfoot}%
\usepackage{enumerate}

\usepackage{mathdots}  

\usepackage{tikz}
\colorlet{myblue}{blue!60!black}
\definecolor{bleudefrance}{rgb}{0.19, 0.55, 0.91}

\usepackage{color}
\usepackage[normalem]{ulem}
\newcommand{\red}[1]{{\color{red} #1}}

\newcommand{\C}{\mathbb C}
\newcommand{\R}{\mathbb R}

\newcommand{\Z}{\mathbb{Z}}
\newcommand{\N}{\mathbb{N}}

\newcommand{\spanset}{\mathrm{span}} 
\newcommand{\supp}{\mathrm{supp\,}} 

\makeatletter
\def\moverlay{\mathpalette\mov@rlay}
\def\mov@rlay#1#2{\leavevmode\vtop{%
   \baselineskip\z@skip \lineskiplimit-\maxdimen
   \ialign{\hfil$\m@th#1##$\hfil\cr#2\crcr}}}
\newcommand{\charfusion}[3][\mathord]{
    #1{\ifx#1\mathop\vphantom{#2}\fi
        \mathpalette\mov@rlay{#2\cr#3}
      }
    \ifx#1\mathop\expandafter\displaylimits\fi}
\makeatother

\def\bn{ \boldsymbol{n}}

\def\bx{ \boldsymbol{x}}

\def\bz{ \boldsymbol{z}}

\def\bgamma{ \boldsymbol{\gamma} }
\def\blambda{ \boldsymbol{\lambda} }
\def\btau{ \boldsymbol{\tau}}

\def\bomega{ \boldsymbol{\omega}}

\theoremstyle{thmstyleone}%
\newtheorem{theorem}{Theorem}
\newtheorem{proposition}[theorem]{Proposition}%
\newtheorem{corollary}[theorem]{Corollary}
\newtheorem{lemma}[theorem]{Lemma}

\newtheorem{definition}[theorem]{Definition}%
\newtheorem{example}[theorem]{Example}%
\newtheorem{remark}[theorem]{Remark}%

\theoremstyle{thmstyletwo}%

\theoremstyle{thmstylethree}%

\begin{document}

\title{Exponential bases for parallelepipeds with frequencies lying in a prescribed lattice}


\author*[1]{\fnm{Dae Gwan} \sur{Lee}}\email{daegwans@gmail.com, daegwan@kumoh.ac.kr} 

\author[2]{\fnm{G\"otz E.} \sur{Pfander}}\email{pfander@ku.de} 

\author[3]{\fnm{David} \sur{Walnut}}\email{dwalnut@gmu.edu} 

\affil*[1]{\orgdiv{Department of Mathematics and Big Data Science}, \orgname{Kumoh National Institute of Technology}, \orgaddress{\street{Daehak-ro~61}, \city{Gumi}, \state{Gyeongsangbuk-do} \postcode{39177}, \country{Republic of Korea (South Korea)}}}

\affil[2]{\orgdiv{Mathematical Institute for Machine Learning and Data Science (MIDS)}, \orgname{Katholische Universit\"at Eichst\"att--Ingolstadt}, \orgaddress{\street{Hohe-Schul-Str.~5}, \postcode{85049} \city{Ingolstadt}, \country{Germany}}}

\affil[3]{\orgdiv{Department of Mathematical Sciences}, \orgname{George Mason University}, \orgaddress{\street{4400~University Drive}, \city{Fairfax}, \state{VA} \postcode{22030}, \country{USA}}}

%
%


\abstract{The existence of a Fourier basis with frequencies in $\mathbb{R}^d$ for the space of square integrable functions supported on a given parallelepiped in $\mathbb{R}^d$, has been well understood since the 1950s. In a companion paper, we derived necessary and sufficient conditions for a parallelepiped in $\mathbb{R}^d$ to permit an orthogonal basis of exponentials with frequencies constrained to be a subset of a prescribed lattice in $\mathbb{R}^d$, a restriction relevant in many applications.
In this paper, we investigate analogous conditions for parallelepipeds that permit a Riesz basis of exponentials with the same constraints on the frequencies. We provide a sufficient condition on the parallelepiped for the Riesz basis case which directly extends one of the necessary and sufficient conditions obtained in the orthogonal basis case. We also provide a sufficient condition which constrains the spectral norm of the matrix generating the parallelepiped, instead of constraining the structure of the matrix.}



\maketitle

\section{Introduction and main results}
\label{sec:intro}

A complex exponential system is expressed as $\mathcal{E}(\Gamma) := \{e^{2\pi i \gamma \cdot (\cdot)} : \gamma \in\Gamma\}$ on a measurable set $S \subseteq \R^d$ called a \emph{domain}, where $\Gamma \subseteq \R^d$ is a discrete set called a \emph{frequency set} or a \emph{spectrum}.
For example, $\mathcal{E}(\Z)$ is an orthogonal basis for $L^2[0,1]$, illustrating the fundamental principle that every $1$-periodic function can be expressed as a weighted sum of sinusoidal waves with integer frequencies, an observation which lies at the heart of Fourier analysis.  
Nowadays, complex exponential systems, particularly exponential frames and bases, are used extensively in many areas of mathematics and engineering. 
For instance, OFDM (orthogonal frequency division multiplexing) based communications involve the design of exponential bases for unions of intervals.  
From an application perspective, it is preferable that the frequency components of the exponentials lie on a prescribed, often uniform rectangular grid such as the integer lattice, due to requirements of the application or hardware constraints.

The goal of this paper is to analyze the existence of orthogonal or Riesz bases of exponentials with frequencies constrained to a prescribed lattice, for domains that are parallelepipeds in $\R^d$.
Formally, we seek to construct a frequency set $\Gamma \subseteq B\Z^d$ such that $\mathcal{E}(\Gamma)$ is an orthogonal basis or a Riesz basis for $L^2( A [0,1]^d )$, where the matrices $A, B \in \R^{d\times d}$ are given a priori.

In a companion paper, we proved the following.


\begin{theorem}[Corollary 7 in \cite{LPW24-first-cubes}]
\label{thm:cubes-tiling-expRB}
For nonsingular matrices $A,B \in \R^{d \times d}$, the following are equivalent:
\begin{enumerate}[(i)]
     \item\label{list:thm1part12} there exists a lattice $\Lambda \subseteq B\Z^d$ with $\mathcal E(\Lambda)$ being an orthogonal basis for $L^2(A [0,1]^d)$;
     \item\label{list:thm1part22-prime}      
     there exist a nonsingular integer matrix $R \in \Z^{d \times d}$ and a lower unitriangular\footnote{A triangular matrix is called unitriangular if its diagonal entries are all equal to $1$.} matrix $G \in \R^{d \times d}$ with $A [0,1]^d = B^{-T} R^{-1} G [0,1]^d$.  
 \end{enumerate} 
 For $d\leq 7$, the statements above are also equivalent to: 
 \begin{enumerate}[(i)]
 \setcounter{enumi}{2}
     \item\label{list:thm1part32} there exists a set $\Phi \subseteq B\Z^d$ with $\mathcal E(\Phi)$ being an orthogonal basis for $L^2(A [0,1]^d)$.
\end{enumerate}   
\end{theorem}

In this paper, we are interested in relaxing the condition of orthogonal bases to Riesz bases,
which will allow for a broader range of feasible matrices $A \in \R^{d \times d}$.  
Our first main result is the following.


\begin{theorem}\label{thm:main2-RB-general}
Assume that $A [0,1]^d = B^{-T} R^{-1} H [0,1]^d$ where $A,B \in \R^{d \times d}$ are nonsingular matrices, $R \in \Z^{d \times d}$ is a nonsingular integer matrix and $H \in \R^{d \times d}$ is a lower triangular matrix with diagonals lying in $(0,1]$. 
Then there exists a set $\Gamma\subseteq B\Z^d$ such that $\mathcal{E}(\Gamma)$ is a Riesz basis for $L^2(A [0,1]^d)$.
\end{theorem}

Setting $B = I_d$ reduces Theorem~\ref{thm:main2-RB-general} to the following.

\begin{theorem}\label{thm:main2-RB-integer}
Assume that $A [0,1]^d = R^{-1} H [0,1]^d$ where $A \in \R^{d \times d}$ is a nonsingular matrix, $R \in \Z^{d \times d}$ is a nonsingular integer matrix and $H \in \R^{d \times d}$ is a lower triangular matrix with diagonals lying in $(0,1]$.  
Then there exists a set $\Gamma\subseteq \Z^d$ such that $\mathcal{E}(\Gamma)$ is a Riesz basis for $L^2(A [0,1]^d)$.
\end{theorem}

This result deals only with frequency sets that are subsets of $\mathbb{Z}^d$.
Nevertheless, as we shall see later, Theorems~\ref{thm:main2-RB-general} and \ref{thm:main2-RB-integer} are in fact equivalent.

Note that the parallelepipeds $A [0,1]^d = B^{-T} R^{-1} G [0,1]^d$ in Theorem~\ref{thm:cubes-tiling-expRB} correspond to a special case of the parallelepipeds $A [0,1]^d = B^{-T} R^{-1} H [0,1]^d$ in Theorem~\ref{thm:main2-RB-general} with all diagonal entries of $H$ set to $1$.   
On the other hand, the structural conditions of the parallelepipeds can be lifted by instead constraining the spectral norm of $A$. 

\begin{theorem}
\label{thm:limiting-spectral-norm}
Let $A \in \R^{d \times d}$ be a nonsingular matrix with spectral norm less than $\frac{2\ln(2)}{\pi d^{3/2}}$, that is, $\| A \|_2 < \frac{2\ln(2)}{\pi d^{3/2}}$. 
If $\mathcal{C}$ is the set obtained by rounding each element of $A^{-T} \, \Z^d$ to its closest point in $\Z^d$, then $\mathcal{E}(\mathcal{C})$ is a Riesz basis for $L^2(A [0,1]^d)$.
\end{theorem}

Note that $\frac{2\ln(2)}{\pi d^{3/2}} \approx 0.1560$ for $d = 2$, and is even smaller for $d \geq 3$.
This is certainly much smaller than the optimal bound $1$ suggested by Theorems~\ref{thm:cubes-tiling-expRB}, \ref{thm:main2-RB-general} and \ref{thm:main2-RB-integer}. 
A proof of Theorem~\ref{thm:limiting-spectral-norm} is given in Section~\ref{sec:proof-main1}. 

\subsection{Related work}
\label{subsec:RelatedWork}

The literature on exponential bases is extensive. 
Below, we give a short review on exponential orthogonal/Riesz bases, focused on the papers that are closely related to our work. 

\subsubsection*{Exponential orthogonal bases}
Fuglede's conjecture \cite{Fu74} states that if $S \subseteq \R^d$ is a set of finite measure, then
there exists an orthogonal basis $\mathcal{E}(\Gamma)$ for $L^2(S)$ (with some $\Gamma \subseteq \R^d$)
if and only if the set $S$ tiles $\R^d$ by translations along a discrete set $\Omega \subseteq \R^d$ in the sense that
\[
\sum_{\gamma \in \Omega} \chi_S (x + \gamma) = 1
\quad \text{for a.e.} \;\; x \in \R^d ,
\]
where $\chi_S$ is the indicator function of $S$ defined by $\chi_S (x) = 1$ for $x \in S$, and $0$ otherwise. 
The conjecture is known to be false for $d \geq 3$ but it remains open for $d = 1,2$.
However, there are some special cases in which the conjecture is known to be true (see \cite{LM19} and the references therein).
For instance, the conjecture is true in all dimensions $d \geq 1$ if $\Gamma$ is a lattice in $\R^d$ \cite{Fu74} (see Proposition~\ref{prop:Fuglede-for-lattices} below), or if $S \subseteq \R^d$ is a convex set of finite positive measure \cite{LM19}.
In particular, when $S$ is the unit ball in $\R^d$, it was shown in \cite{IKP99} that no exponential orthogonal basis exists for $L^2(S)$ if $d \geq 2$, while in the $1$-dimensional case it is clear that $\mathcal{E}(\frac{1}{2}\Z)$ is an orthogonal basis for $L^2 [-1,1]$.
It was also shown that for any discrete set $\Gamma \subseteq \R^d$, the system $\mathcal{E}(\Gamma)$ is an orthogonal basis for $L^2[0,1]^d$ if and only if $\{ \gamma + [0,1]^d : \gamma \in \Gamma \}$ is a tiling of $\R^d$ \cite{JP99}. 
In fact, the cube $[0,1]^d$ in this statement can be replaced with any open set of measure $1$ \cite{Ko00}. 
Moreover, it was shown that if $S = A [0,1]^d$ with $A \in \mathrm{GL} (d,\R)$, then for a discrete set $\Gamma \subseteq \R^d$, the system $\mathcal{E}(\Gamma)$ is an orthogonal basis for $L^2(S)$ if and only if $\{ \gamma +  A^{-T} \, [0,1]^d : \gamma \in \Gamma \}$ is a tiling of $\R^d$ \cite{LRW00}.
For more details on recent progress on Fuglede's conjecture, see \cite{LM19} and the reference therein.

Han and Yang \cite{HW01} proved that
given any two lattices of the same density in $\R^d$, there exists a measurable set which is a fundamental domain for both lattices. This insight was central in seeing the possibility of finding non-obvious frequency sets for parallelepipeds.

\medskip

\subsubsection*{Exponential Riesz bases}
The problem considered here is closely related to so-called basis extraction due to Avdonin (\cite{Av74, AI95}, see also \cite{Se95}): 
if $\mathcal{E}(\Gamma)$ is a Riesz basis for $L^2[0,T]$ for some $\Gamma \subseteq \R$ and $T > 0$, then for each $0 < T' < T$ there exists a subset $\Gamma' \subseteq \Gamma$ such that $\mathcal{E}(\Gamma')$ is a Riesz basis for $L^2[0,T']$ (see Theorem II.4.16 in \cite{AI95}).  
In our case, if $A[0,1]^d\subseteq[0,1]^d$ then we are extracting a basis for $L^2(A[0,1]^d)$ from the basis $\mathcal{E}(\Z^d)$ of $L^2[0,1]^d$.

In 1995, Seip \cite{Se95} showed that for any interval $I \subseteq [0,1)$, there exists a set $\Gamma \subseteq \Z$ such that $\mathcal{E}(\Gamma)$ is a Riesz basis for $L^2(I)$.
In 2015, Kozma and Nitzan \cite{KN15} proved that if $S$ is a finite union of disjoint intervals in $[0,1)$, then there exists a set $\Gamma \subseteq \Z$ such that $\mathcal{E}(\Gamma)$ is a Riesz basis for $L^2(S)$.
So far, there are relatively few classes of sets for which exponential Riesz bases are known to exist, see e.g., \cite{CC18,CL22,DL19,Ko15,KN15,KN16,Le22,LPW23,PRW21}.
While most of the sets in $\R^d$ are deemed to admit exponential Riesz bases, it was shown recently that for a certain bounded measurable set $S \subseteq \R$, no set of exponentials can be a Riesz basis for $L^2(S)$ \cite{KNO21}.

\subsection{Organization of the paper}
\label{subsec:organization-paper}

In Section \ref{sec:prelim}, we review some background material including perturbation theorems, Fuglede's theorem and Paley--Wiener spaces.
In Section~\ref{sec:proof-thms-equivalence}, we provide a proof for the equivalence of Theorems~\ref{thm:main2-RB-general} and \ref{thm:main2-RB-integer}. 
Sections~\ref{sec:proof-thm-main2-RB} and \ref{sec:proof-main1} are devoted to the proof of Theorems~\ref{thm:main2-RB-integer} and \ref{thm:limiting-spectral-norm}, respectively.
In order to keep the paper self-contained, we provide a proof of Theorem~\ref{thm:tensor-product-Kadec-Avdonin}, Lemma~\ref{lem:RB-basic-operations} and Proposition~\ref{lem:a-irrational-Beatty-Fraenkel} in Appendix~\ref{sec:proof-thm:tensor-product-Kadec-Avdonin}, Appendix~\ref{sec:proof-lem:RB-basic-operations} and Appendix~\ref{sec:proof-lem:a-irrational-Beatty-Fraenkel}, respectively.

\section{Preliminaries}
\label{sec:prelim}

\subsection{Riesz Bases and perturbation theorems}

A Riesz basis is a generalization of an orthogonal basis which relaxes the condition of orthogonality while still maintaining some desirable properties---a Riesz basis is complete and minimal but not necessarily orthogonal.

\begin{definition}[see Definition 3.6.1 and Theorem 3.6.6 in \cite{Ch16}, or Definition 7.9 and Theorem 8.32 in \cite{He11}]
\label{def:RieszBases}
Let $\{ e_n \}_{n\in\Z}$ be an orthonormal basis for a separable Hilbert space $\mathcal{H}$.
A sequence $\{ f_n \}_{n\in\Z} \subseteq \mathcal{H}$ is a \emph{Riesz basis} for $\mathcal{H}$ if
there exists a bijective bounded operator $U : \mathcal{H} \rightarrow \mathcal{H}$ such that $U e_n = f_n$ for all $n \in \Z$.
Equivalently, a sequence $\{ f_n \}_{n\in\Z} \subseteq \mathcal{H}$ is a \emph{Riesz basis} for $\mathcal{H}$ if it is complete in $\mathcal{H}$, i.e., $\overline{\spanset} \, \{ f_n \}_{n\in\Z} = \mathcal{H}$, and
if there are constants $0 < A \leq B < \infty$ such that
\begin{equation}\label{eqn:HilbertSpaceRieszCondition}
A \, \sum_{n \in \Z} | c_n |^2
\leq \Big\| \sum_{n\in\Z} c_n \, f_n \Big\|^2
\leq B \, \sum_{n \in \Z} | c_n |^2
\quad \text{for all} \;\;  \{c_n\}_{n\in\Z} \in \ell_2 (\mathbb Z) .
\end{equation}
A sequence $\{ f_n \}_{n\in\Z} \subseteq \mathcal{H}$ is called a \emph{Riesz sequence} if it satisfies \eqref{eqn:HilbertSpaceRieszCondition} for some $0 < A \leq B < \infty$, and a \emph{Bessel sequence} if it satisfies the right inequality of \eqref{eqn:HilbertSpaceRieszCondition} for some $B > 0$.
\end{definition}


The set of exponentials $\mathcal{E}(\Z) = \{e^{2\pi i n (\cdot)} : n \in \Z \}$ is an orthogonal basis for $L^2[0,1]$, and if each integer is perturbed slightly then one still obtains a Riesz basis.

\begin{proposition}[Kadec's $\frac{1}{4}$-theorem \cite{Ka64}, see also Theorem 14 on p.\,36 in \cite{Yo01}]
\label{prop:Kadec-1-4}
Let $\{ \gamma_n \}_{n \in \Z} \subseteq \R$ be a sequence satisfying
\begin{equation}\label{eqn:Kadec-condition}
\sup_{n \in \Z} | \gamma_n - n |
\;\leq\; L
\;<\; \tfrac{1}{4} .
\end{equation}
Then $\mathcal{E}( \{ \gamma_n \}_{n \in \Z} )$ is a Riesz basis for $L^2[0,1]$ with bounds $(1 - B(L))^2$ and $(1 + B(L))^2$, where $B(L) = 1 - \cos (\pi L) + \sin (\pi L)$.
\end{proposition}

The condition \eqref{eqn:Kadec-condition} means that $\{ \gamma_n \}_{n \in \Z}$ is a pointwise perturbation of $\Z$ up to $L \; (<  \frac{1}{4})$.
Avdonin \cite{Av74} weakened the condition to ``perturbation in the mean''.

\begin{proposition}[Avdonin's theorem of ``$\frac{1}{4}$ in the mean'' \cite{Av74}, also see the notes on p.\,178 in \cite{Yo01}]
\label{prop:Avdonin-1-4-in-the-mean}
Let $\{ \gamma_n = n + \delta_n \}_{n \in \Z} \subseteq \R$ be a separated sequence, that is, $\inf \{ |\gamma_n - \gamma_{n'}| : n,n'\in\Z ,  \; n \neq n' \} > 0$.
If there exists an integer $P \in \N$ with
\begin{equation}\label{eqn:Avdonin-condition}
\sup_{m \in \Z}
\tfrac{1}{N}
\Big| \sum_{k=mP}^{(m+1)P-1} \delta_k \Big|
\;\leq\; L
\;<\; \tfrac{1}{4} ,
\end{equation}
then $\mathcal{E}( \{ \gamma_n \}_{n \in \Z} )$ is a Riesz basis for $L^2[0,1]$.
\end{proposition}

In higher dimensions, we can construct exponential Riesz bases for $L^2[0,1]^d$ by taking a product of the sets $\{ \gamma_n \}_{n \in \Z}$ from Propositions \ref{prop:Kadec-1-4} and \ref{prop:Avdonin-1-4-in-the-mean}.

\begin{theorem}[cf.~Theorem 2.1 in \cite{SZ99-ACHA}, Theorem 1.1 in \cite{SZ99-JMAA}]
\label{thm:tensor-product-Kadec-Avdonin}
Let $d \in \N$.
For each $k = 1, \ldots, d$, let $\Gamma^{(k)} = \{ \gamma_n^{(k)} \}_{n \in \Z} \subseteq \R$ be a sequence satisfying \eqref{eqn:Kadec-condition} or \eqref{eqn:Avdonin-condition}.
Then $\mathcal{E}( \Gamma^{(1)} {\times} \ldots {\times} \Gamma^{(d)} )$ is a Riesz basis for $L^2[0,1]^d$.
Moreover, if all the sets $\Gamma^{(k)} = \{ \gamma_n^{(k)} \}_{n \in \Z}$ satisfy \eqref{eqn:Kadec-condition}, that is, if $L^{(k)} := \sup_{n \in \Z} | \gamma_n^{(k)} - n | < \frac{1}{4}$ for all $k = 1, \ldots, d$, then the Riesz bounds can be chosen to be $\prod_{k=1}^d (1 - B(L^{(k)}))^2$ and $\prod_{k=1}^d (1 + B(L^{(k)}))^2$, where $B(L) = 1 - \cos (\pi L) + \sin (\pi L) < 1$ for $0 \leq L < \frac{1}{4}$.
\end{theorem}

A proof of Theorem~\ref{thm:tensor-product-Kadec-Avdonin} is given in Appendix~\ref{sec:proof-thm:tensor-product-Kadec-Avdonin}.

The frequency set $\Gamma^{(1)} {\times} \ldots {\times} \Gamma^{(d)}$ in the statement of Theorem \ref{thm:tensor-product-Kadec-Avdonin} is a \emph{tensor product} which is often too restrictive.
In fact, one can also construct exponential Riesz bases for $L^2[0,1]^d$ by choosing a frequency set $\Gamma \subseteq \R^d$ which is a pointwise perturbation of $\Z^d$, meaning that $\Gamma$ is obtained by perturbing each element of $\Z^d$ independently.
Certainly, such frequency sets are much more flexible than tensor product sets.

\begin{theorem}[Corollary 6.1 in \cite{Ba10}; the case $d=1$ was proved by Duffin and Eachus \cite{DE42}, see also Remarks on p.\,37 in \cite{Yo01}]
\label{thm:Bailey-Cor6-1}
For each $\bn = (n^{(1)}, \ldots, n^{(d)}) \in \Z^d$, let $\bgamma_{\bn} = (\gamma_{\bn}^{(1)}, \ldots, \gamma_{\bn}^{(d)}) \in \R^d$ be a perturbation of $\bn$ such that
\begin{equation}\label{eqn:Bailey-Cor6-1-condition-ln2-over-pid}
L := \sup_{\bn \in \Z^d} \| \bgamma_{\bn} - \bn \|_{\infty}
= \sup_{\bn \in \Z^d} \max_{1 \leq k \leq d} | \gamma_{\bn}^{(k)} - n^{(k)} | \;<\; \tfrac{\ln(2)}{\pi d} .
\end{equation}
Then $\mathcal{E}( \{ \bgamma_{\bn} : \bn \in \Z^d \} )$ is a Riesz basis for $L^2[0,1]^d$.
\end{theorem}

It is worth noting that in dimension $d=1$, the constant $\frac{\ln(2)}{\pi} \; (\approx 0.2206)$ appearing in \eqref{eqn:Bailey-Cor6-1-condition-ln2-over-pid} is smaller than $\frac{1}{4}$ and therefore Theorem \ref{thm:Bailey-Cor6-1} is weaker than Theorem \ref{thm:tensor-product-Kadec-Avdonin} in dimension $1$; in fact, it is known that $\frac{1}{4}$ is the optimal constant (see Remarks on p.\,37 in \cite{Yo01}, also see p.\,202 in \cite{He11}).
For higher dimensions $d \geq 2$, the constant $\frac{\ln(2)}{\pi d}$ in \eqref{eqn:Bailey-Cor6-1-condition-ln2-over-pid} is much smaller than $\frac{1}{4}$ but the frequency sets satisfying \eqref{eqn:Bailey-Cor6-1-condition-ln2-over-pid} have more flexible structure than tensor product sets; hence, both Theorems \ref{thm:tensor-product-Kadec-Avdonin} and \ref{thm:Bailey-Cor6-1} have pros and cons over the other.
See Figure \ref{fig:comparison-of-frequency-set-structure} for a comparison of the frequency sets apprearing in Theorems \ref{thm:tensor-product-Kadec-Avdonin} and \ref{thm:Bailey-Cor6-1}.

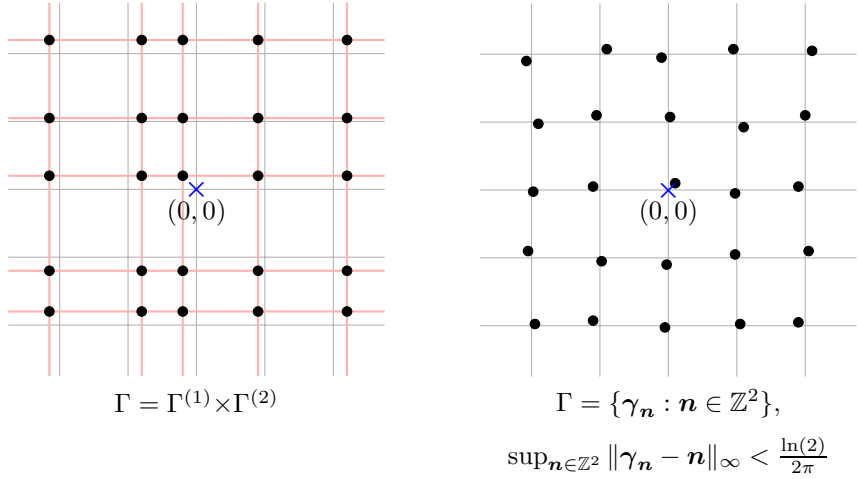
\begin{figure}[h!]
\begin{center}
\begin{tikzpicture}[scale=0.9]
    \draw[black!30] (-2.75,-2.75) grid (2.75,2.75);
	\draw[thick,red!30] ({-2 - 0.15},-2.75) -- ({-2 - 0.15},2.75);
	\draw[thick,red!30] ({-1 + 0.2},-2.75) -- ({-1 + 0.2},2.75);    	
	\draw[thick,red!30] ({0 - 0.2},-2.75) -- ({0 - 0.2},2.75);
	\draw[thick,red!30] ({1 - 0.1},-2.75) -- ({1 - 0.1},2.75);
	\draw[thick,red!30] ({2 + 0.2},-2.75) -- ({2 + 0.2},2.75);

	\draw[thick,red!30] (-2.75, {2 + 0.2}) -- (2.75, {2 + 0.2});
	\draw[thick,red!30] (-2.75, {1 + 0.05}) -- (2.75, {1 + 0.05});       	
	\draw[thick,red!30] (-2.75, {0 + 0.2}) -- (2.75, {0 + 0.2});
	\draw[thick,red!30] (-2.75, {-1 - 0.2}) -- (2.75, {-1 - 0.2});
	\draw[thick,red!30] (-2.75, {-2 + 0.2}) -- (2.75, {-2 + 0.2});
				
        \draw ({-2 - 0.15},{2 + 0.2}) circle (0.7mm) [fill=black];
        \draw ({-1 + 0.2},{2 + 0.2}) circle (0.7mm) [fill=black];
        \draw ({0 - 0.2},{2 + 0.2}) circle (0.7mm) [fill=black];
        \draw ({1 - 0.1},{2 + 0.2}) circle (0.7mm) [fill=black];
        \draw ({2 + 0.2},{2 + 0.2}) circle (0.7mm) [fill=black];

        \draw ({-2 - 0.15},{1 + 0.05}) circle (0.7mm) [fill=black];
        \draw ({-1 + 0.2},{1 + 0.05}) circle (0.7mm) [fill=black];
        \draw ({0 - 0.2},{1 + 0.05}) circle (0.7mm) [fill=black];
        \draw ({1 - 0.1},{1 + 0.05}) circle (0.7mm) [fill=black];
        \draw ({2 + 0.2},{1 + 0.05}) circle (0.7mm) [fill=black];

        \draw ({-2 - 0.15},{0 + 0.2}) circle (0.7mm) [fill=black];
        \draw ({-1 + 0.2},{0 + 0.2}) circle (0.7mm) [fill=black];
        \draw ({0 - 0.2},{0 + 0.2}) circle (0.7mm) [fill=black];
        \draw ({1 - 0.1},{0 + 0.2}) circle (0.7mm) [fill=black];
        \draw ({2 + 0.2},{0 + 0.2}) circle (0.7mm) [fill=black];

        \draw ({-2 - 0.15},{-1 - 0.2}) circle (0.7mm) [fill=black];
        \draw ({-1 + 0.2},{-1 - 0.2}) circle (0.7mm) [fill=black];
        \draw ({0 - 0.2},{-1 - 0.2}) circle (0.7mm) [fill=black];
        \draw ({1 - 0.1},{-1 - 0.2}) circle (0.7mm) [fill=black];
        \draw ({2 + 0.2},{-1 - 0.2}) circle (0.7mm) [fill=black];

        \draw ({-2 - 0.15},{-2 + 0.2}) circle (0.7mm) [fill=black];
        \draw ({-1 + 0.2},{-2 + 0.2}) circle (0.7mm) [fill=black];
        \draw ({0 - 0.2},{-2 + 0.2}) circle (0.7mm) [fill=black];
        \draw ({1 - 0.1},{-2 + 0.2}) circle (0.7mm) [fill=black];
        \draw ({2 + 0.2},{-2 + 0.2}) circle (0.7mm) [fill=black];
\node[below] at (0,0) {$(0,0)$};
\node[blue] at (0,0) {$\boldsymbol{\times}$};
\node[below] at (0,-2.8) {$\Gamma = \Gamma^{(1)} {\times} \Gamma^{(2)}$};
\node[below] at (0,-4.04) {$ $};
\end{tikzpicture}
\hspace{10mm}
\begin{tikzpicture}[scale=0.9]
    \draw[black!30] (-2.75,-2.75) grid (2.75,2.75);
				
        \draw ({-2 - 0.15/2},{2 - 0.2/2}) circle (0.7mm) [fill=black];
        \draw ({-1 + 0.2/2},{2 + 0.15/2}) circle (0.7mm) [fill=black];
        \draw ({0 - 0.2/2},{2 - 0.1/2}) circle (0.7mm) [fill=black];
        \draw ({1 - 0.1/2},{2 + 0.15/2}) circle (0.7mm) [fill=black];
        \draw ({2 + 0.2/2},{2 + 0.1/2}) circle (0.7mm) [fill=black];

        \draw ({-2 + 0.2/2},{1 - 0.05/2}) circle (0.7mm) [fill=black];
        \draw ({-1 - 0.1/2},{1 + 0.2/2}) circle (0.7mm) [fill=black];
        \draw ({0 + 0.05/2},{1 + 0.15/2}) circle (0.7mm) [fill=black];
        \draw ({1 + 0.2/2},{1 - 0.15/2}) circle (0.7mm) [fill=black];
        \draw ({2 + 0/2},{1 + 0.2/2}) circle (0.7mm) [fill=black];

        \draw ({-2 + 0.05/2},{0 - 0.-0.05/2}) circle (0.7mm) [fill=black];
        \draw ({-1 - 0.2/2},{0 + 0.1/2}) circle (0.7mm) [fill=black];
        \draw ({0 + 0.2/2},{0 + 0.2/2}) circle (0.7mm) [fill=black];
        \draw ({1 - 0.05/2},{0 - 0.1/2}) circle (0.7mm) [fill=black];
        \draw ({2 - 0.2/2},{0 + 0.1/2}) circle (0.7mm) [fill=black];

        \draw ({-2 - 0.1/2},{-1 + 0.2/2}) circle (0.7mm) [fill=black];
        \draw ({-1 + 0.05/2},{-1 - 0.1/2}) circle (0.7mm) [fill=black];
        \draw ({0 - 0.05/2},{-1 - 0.2/2}) circle (0.7mm) [fill=black];
        \draw ({1 - 0.05/2},{-1 + 0.1/2}) circle (0.7mm) [fill=black];
        \draw ({2 + 0.1/2},{-1 + 0.2/2}) circle (0.7mm) [fill=black];

        \draw ({-2 + 0.1/2},{-2 + 0.05/2}) circle (0.7mm) [fill=black];
        \draw ({-1 - 0.2/2},{-2 + 0.15/2}) circle (0.7mm) [fill=black];
        \draw ({0 - 0.1/2},{-2 - 0.05/2}) circle (0.7mm) [fill=black];
        \draw ({1 + 0.1/2},{-2 + 0.05/2}) circle (0.7mm) [fill=black];
        \draw ({2 - 0.2/2},{-2 + 0.1/2}) circle (0.7mm) [fill=black];
\node[below] at (0,0) {$(0,0)$};
\node[blue] at (0,0) {$\boldsymbol{\times}$};
\node[below] at (0,-2.8) {$\Gamma = \{ \bgamma_{\bn} : \bn \in \Z^2 \}$, };
\node[below] at (0,-3.5) {$\sup_{\bn \in \Z^2} \| \bgamma_{\bn} - \bn \|_{\infty} < \frac{\ln(2)}{2 \pi }$ };
\end{tikzpicture}
\end{center}
\caption{Both of the frequency sets in Theorem \ref{thm:tensor-product-Kadec-Avdonin} (left) and Theorem \ref{thm:Bailey-Cor6-1} (right) are perturbations of $\Z^2$. The left set is a tensor product while the right set is obtained by independent perturbation of points in $\Z^2$. The size of perturbation allowed in each coordinate direction is less than $1/4 = 0.25$ for the left set, and is less than $\frac{\ln(2)}{2 \pi } \approx 0.1103$ for the right set.}
\label{fig:comparison-of-frequency-set-structure}
\end{figure}

\begin{remark}
\rm
In \cite[Theorem 2.1]{SZ99-ACHA} and \cite[Theorem 1.1]{SZ99-JMAA}, the authors claimed the stronger result that the constant $\frac{\ln(2)}{\pi d}$ in \eqref{eqn:Bailey-Cor6-1-condition-ln2-over-pid} can be replaced by $\frac{1}{4}$ for all dimensions $d \in \N$.
However, their proof actually shows that
$\mathcal{E}( \Gamma^{(1)} {\times} \ldots {\times} \Gamma^{(d)} )$ is a Riesz basis for $L^2[0,1]^d$ if each $\Gamma^{(k)} \subseteq \R$ satisfies the Kadec $\frac{1}{4}$-condition given by \eqref{eqn:Kadec-condition}.
Theorem \ref{thm:tensor-product-Kadec-Avdonin} is an improvement of their result
in the sense that each $\Gamma^{(k)}$ is allowed to satisfy either Kadec's $\frac{1}{4}$-condition \eqref{eqn:Kadec-condition} or Avdonin's ``$\frac{1}{4}$ in the mean'' condition \eqref{eqn:Avdonin-condition}.
\end{remark}

\begin{remark}
\rm 
The perturbation theorems introduced in this section, produce frequency sets that are perturbed away from a lattice, especially from the integer lattice $\Z^d$. 
In contrast, Theorems~\ref{thm:main2-RB-integer} and \ref{thm:limiting-spectral-norm} deal with frequency sets that are subsets of $\Z^d$.
\end{remark}

\subsection{Density of frequency sets and invariance of exponential Riesz bases}
\label{subsec:density-RBinvariance}

The lower and upper density of a discrete set $\Gamma \subseteq \R^d$ is defined by
\[
\begin{split}
& D^-(\Gamma) = \liminf_{r \rightarrow \infty} \frac{\inf_{\mathbf{z} \in \R^d} |\Gamma \cap (\mathbf{z} + [0,r]^d) |}{r^d} \; , \\
& D^+(\Gamma) = \limsup_{r \rightarrow \infty} \frac{\sup_{\mathbf{z} \in \R^d} |\Gamma \cap (\mathbf{z} + [0,r]^d) |}{r^d} \; , 
\end{split}
\]
respectively (see e.g., \cite{He07}). If $D^-(\Gamma) = D^+(\Gamma)$, then $\Gamma$ is said to have \emph{uniform} density $D(\Gamma) := D^-(\Gamma) = D^+(\Gamma)$. 

\begin{proposition}[Landau's density theorem \cite{La67}] 
\label{prop:Landau}
Let $\Gamma \subseteq \R^d$ be a discrete set and let $S \subseteq \R^d$ be a bounded set.
If $\mathcal{E}(\Gamma)$ is a Riesz basis for $L^2(S)$, then $\Gamma$ has uniform density and $D(\Gamma) = |S|$.
\end{proposition}

Let $\mathrm{GL} (d,\R)$ be the \emph{general linear group of degree $d$} which consists of $d {\times} d$ nonsingular real matrices. 

\begin{lemma}\label{lem:RB-basic-operations}
Assume that $\mathcal{E}(\Gamma)$ is a Riesz basis for $L^2(S)$ with bounds $0 < \alpha \leq \beta < \infty$, where $\Gamma \subseteq \R^d$ is a discrete set and $S \subseteq \R^d$ is a measurable set. \\
({\romannumeral 1}) For any $a \in \R^d$, the system $\mathcal{E}(\Gamma)$ is a Riesz basis for $L^2(S+a)$ with bounds $\alpha$ and $\beta$. \\
({\romannumeral 2}) For any $b \in \R^d$, the system $\mathcal{E}(\Gamma + b)$ is a Riesz basis for $L^2(S)$ with bounds $\alpha$ and $\beta$. \\
({\romannumeral 3}) For any $A \in \mathrm{GL} (d,\R)$, the system $\mathcal{E}(A \Gamma)$ is a Riesz basis for $L^2(A^{-T} S)$ with bounds $\frac{\alpha}{| \! \det (A) |} $ and $\frac{\beta}{| \! \det (A) |} $.
In particular, for any $c > 0$, the system $\mathcal{E}(c \Gamma)$ is a Riesz basis for $L^2(\frac{1}{c} S)$ with bounds $c^{-d} \alpha$ and $c^{-d} \beta$. 
\end{lemma}

A proof of Lemma~\ref{lem:RB-basic-operations} is given in Appendix~\ref{sec:proof-lem:RB-basic-operations}.

\subsection{Lattices and Fuglede's theorem}
\label{subsec:lattices-Fuglede}

A \emph{(full rank) lattice} in $\R^d$ is a discrete set given by $\Lambda = A \Z^d$ for some full rank matrix $A \in \R^{d \times d}$, i.e., $A \in \mathrm{GL} (d,\R)$.
The set $A [0,1]^d$ is the parallelepiped (parallelogram if $d=2$) with edges formed by the columns of $A$, and its volume (area if $d=2$) is given by $|\det A|$.
The \emph{dual lattice} of $\Lambda$ is defined by
\[
\begin{split}
\Lambda^*
:=\;& \{ \bgamma \in \R^d : e^{2 \pi i \blambda \cdot \bgamma} = 1 \;\; \text{for all} \; \blambda \in \Lambda \} \\
=\;& \{ \bgamma \in \R^d : \blambda \cdot \bgamma \in \Z \;\; \text{for all} \; \blambda \in \Lambda \} = A^{-T} \Z^d ,
\end{split}
\]
where $A^{-T} := (A^{-1})^T = (A^T)^{-1}$.
In particular, $(c_1 \Z \times \cdots \times c_d \Z)^* = \frac{1}{c_1} \Z \times \cdots \times \frac{1}{c_d} \Z$ for any $c_1, \ldots, c_d > 0$.
A set $S \subseteq \R^d$ is called a \emph{fundamental domain} of $\Lambda$ if for each $\bx\in\R^d$ the set $S$ contains exactly one point from the set $\bx+\Lambda$, equivalently, if 
the sets $S+\blambda$, $\blambda\in\Lambda$, form a partition of $\R^d$.
In particular, the set $A[0,1)^d$ is a fundamental domain of the lattice $A\Z^d$. 

\begin{proposition}[Fuglede's theorem for lattices \cite{Fu74,Io07}]
\label{prop:Fuglede-for-lattices}
Let $\Lambda = A \Z^d$ be a lattice with $A \in \mathrm{GL} (d,\R)$.
Then a set $S \subseteq \R^d$ is a fundamental domain of $\Lambda$ if and only if $\mathcal{E}(\Lambda^*)$ is an orthogonal basis for $L^2(S)$.
In particular, $\mathcal{E}(A^{-T} \Z^d)$ is an orthogonal basis for $L^2(A [0,1)^d)$.
\end{proposition}

The last part of Proposition~\ref{prop:Fuglede-for-lattices} is a generalization of the fact that $\mathcal{E}(\Z^d)$ is an orthogonal basis for $L^2[0,1)^d$; see also Lemma~\ref{lem:RB-basic-operations} ({\romannumeral 3}). 

Finally, note that the spaces $L^2(S_1)$ and $L^2(S_2)$ are identical if $(S_1 \backslash S_2) \cup (S_2 \backslash S_1)$ is a measure zero set. In particular, we have $L^2[0,1)^d = L^2[0,1]^d$ and $L^2(A [0,1)^d) = L^2(A [0,1]^d)$, which allows us to use $[0,1]^d$ and $A [0,1]^d$ for notational convenience.

\subsection{The Fourier transform and Paley--Wiener spaces}

The \emph{$d$-dimensional Fourier transform} is defined as
\[
\mathcal{F}_d [f] (\boldsymbol \omega) = \widehat{f} (\boldsymbol \omega) = \int_{\R^d} f(\boldsymbol t) \, e^{- 2 \pi i \boldsymbol{t} \cdot \boldsymbol{\omega}} \, dt ,
\quad f \in L^1(\R^d) \cap L^2(\R^d) ,
\]
so that $\mathcal{F}_d [\cdot]$ extends to a unitary operator from
$L^{2}(\mathbb{R}^d)$ onto $L^{2}(\mathbb{R}^d)$.
The \emph{Paley--Wiener space} over a measurable set $S \subseteq \R^d$ is defined as
\[
PW(S) := \{ f \in L^2(\R^d) : \supp \widehat{f} \subseteq S \}
\;=\; \mathcal{F}_d ^{-1} \big[ L^2(S) \big]
\]
equipped with the norm $\| f \|_{PW(S)} := \| f \|_{L^2(\R^d)} = \| \widehat{f} \|_{L^2(S)}$.
If $S \subseteq \R^d$ has finite measure, then every $f \in PW(S)$ is continuous and
\begin{equation}\label{eqn:relation-between-PW-and-L2}
f(\boldsymbol x) = \int_{S} \widehat{f} (\boldsymbol \omega) \, e^{2 \pi i \boldsymbol x \cdot \boldsymbol \omega} \, d \boldsymbol \omega
= \big\langle \widehat{f} , e^{- 2 \pi i \boldsymbol x \cdot  (\cdot)} \big\rangle_{L^2(S)}
\quad \text{for all} \;\; \boldsymbol x \in \R^d .
\end{equation}

Let $\Gamma \subseteq \R^d$ be a discrete set and let $S \subseteq \R^d$ be a set of positive measure.
If $\mathcal{E}(\Gamma)$ is a Riesz sequence in $L^2(S)$, then $\Gamma$ is necessarily separated, that is, $\inf\{ |\gamma - \gamma'| : \gamma \neq \gamma' \in \Gamma \} > 0$ (see e.g., \cite[Proposition 11]{Le21}).
Conversely, if $\Gamma \subseteq \R^d$ is separated and $S \subseteq \R^d$ is bounded, then $\mathcal{E}(\Gamma)$ is a Bessel sequence in $L^2(S)$ (see \cite[p.\,135, Theorem 4]{Yo01}).

The theory of exponential bases is closely related with sampling theory. 
Let us introduce some relevant notions.

\begin{definition}
\label{def:set-uniqueness-interpolation}
Let $S \subseteq \R^d$ be a measurable set.
A discrete set $\Gamma \subseteq \R^d$ is called

\begin{itemize}
\item
\emph{a set of uniqueness} for $PW(S)$ if the only function $f \in  PW(S)$ satisfying $f(\gamma) = 0$ for all $\gamma \in \Gamma$ is the trivial function $f=0$;

\item
\emph{a set of sampling} for $PW(S)$
if there are constants $0 < \alpha \leq \beta < \infty$ such that
\[
\alpha \, \| f \|_{PW(S)}^2
\;\leq\;
\sum_{\gamma \in \Gamma} \big| f(\gamma) \big|^2
\;\leq\;
\beta \, \| f \|_{PW(S)}^2
\quad \text{for all} \;\; f \in PW(S) ;
\]

\item
\emph{a set of interpolation} for $PW(S)$ if for each $\{ c_ \gamma \}_{\gamma \in \Gamma} \in \ell_2(\Gamma)$ there exists a function $f \in  PW(S)$ satisfying $f(\gamma) = c_\gamma$ for all $\gamma \in \Gamma$.
\end{itemize}
\end{definition}

Using the relation \eqref{eqn:relation-between-PW-and-L2}, we can easily deduce the following.
\begin{enumerate}[(i)]
\item
$\Gamma$ is a set of uniqueness for $PW(S)$ if and only if $\mathcal{E}(-\Gamma) = \overline{\mathcal{E}(\Gamma)}$ is complete in $L^2(S)$ if and only if $\mathcal{E}(\Gamma)$ is complete in $L^2(S)$.

\item
$\Gamma$ is a set of sampling for $PW(S)$ if and only if $\mathcal{E}(-\Gamma) = \overline{\mathcal{E}(\Gamma)}$ is a frame for $L^2(S)$ if and only if $\mathcal{E}(\Gamma)$ is a frame for $L^2(S)$.
In general, a sequence $\{ f_n \}_{n\in\Z}$ in a separable Hilbert space $\mathcal{H}$ is called a \emph{frame} for $\mathcal{H}$ if there are constants $0 < \alpha \leq \beta < \infty$ such that
\[
\alpha \, \| f \|^2
\;\leq\;
\sum_{n\in\Z} |\langle f , f_n\rangle|^{2}
\;\leq\;
\beta \, \| f \|^2
\quad \text{for all} \;\; f \in \mathcal{H} .
\]

\item
If $\mathcal{E}(\Gamma)$ is a Bessel sequence in $L^2(S)$, then $\Gamma$ is a set of interpolation for $PW(S)$ if and only if $\mathcal{E}(-\Gamma) = \overline{\mathcal{E}(\Gamma)}$ is a Riesz sequence in $L^2(S)$ if and only if $\mathcal{E}(\Gamma)$ is a Riesz sequence in $L^2(S)$ (see \cite[Appendix A]{Le21} for more details).
\end{enumerate}

It is well-known that a sequence in a separable Hilbert space $\mathcal{H}$ is a Riesz basis if and only if it is both a frame and a Riesz sequence (see e.g., \cite[Proposition 3.7.3, Theorems 5.4.1 and 7.1.1]{Ch16}).
Hence, if $\Gamma \subseteq \R^d$ is a separated set and $S \subseteq \R^d$ is a bounded set, then
$\mathcal{E}(\Gamma)$ is a Riesz basis for $L^2(S)$ if and only if $\Gamma$ is a set of uniqueness and interpolation for $PW(S)$
if and only if $\Gamma$ is a set of sampling and interpolation for $PW(S)$.

\section{Proof of the equivalence of Theorems~\ref{thm:main2-RB-general} and \ref{thm:main2-RB-integer}}
\label{sec:proof-thms-equivalence}

Theorem~\ref{thm:main2-RB-general} $\Rightarrow$ Theorem~\ref{thm:main2-RB-integer}: This is obvious by setting $B = I_d$ in Theorem~\ref{thm:main2-RB-general}.

\noindent
Theorem~\ref{thm:main2-RB-integer} $\Rightarrow$ Theorem~\ref{thm:main2-RB-general}: 
Assume that $A [0,1]^d = B^{-T} R^{-1} H [0,1]^d$ where $R \in \Z^{d \times d}$ is a nonsingular integer matrix and $H \in \R^{d \times d}$ is a lower triangular matrix with diagonals lying in $(0,1]$.  
Then $B^T A [0,1]^d = R^{-1} H [0,1]^d$, so Theorem~\ref{thm:main2-RB-integer} provides a set $\Gamma\subseteq \Z^d$ with $\mathcal{E}(\Gamma)$ being a Riesz basis for $L^2(B^T A [0,1]^d)$.
Using Lemma~\ref{lem:RB-basic-operations}, we deduce that $\widetilde{\Gamma} := B \Gamma \subseteq B\Z^d$ is a set with $\mathcal{E}(\widetilde{\Gamma})$ being a Riesz basis for $L^2(B^{-T} B^T A [0,1]^d) = L^2(A [0,1]^d)$, which is the conclusion of Theorem~\ref{thm:main2-RB-general}.
\hfill $\Box$ 

\section{Proof of Theorem~\ref{thm:main2-RB-integer}}
\label{sec:proof-thm-main2-RB}

Define the function $\mathrm{r} : \R \rightarrow \Z$ by $\mathrm{r}(x) = \lfloor x + \frac{1}{2} \rfloor$, where $\lfloor x \rfloor$ denotes the greatest integer less or equal to $x$. This is the customary rounding of $x$ to its nearest integer, namely the \emph{rounding half up}; for instance, $\mathrm{r}(-1.6) = -2$, $\mathrm{r}(-1.5) = -1$, $\mathrm{r}(-1.4) = -1$, $\mathrm{r}(1.4) = 1$, $\mathrm{r}(1.5) = 2$ and $\mathrm{r}(1.6) = 2$. 
We also define
\[
\mathbf{r} : \R^d \rightarrow \Z^d, \quad \mathbf{r} (x_1, \ldots, x_d) := \big( \mathrm{r}(x_1) , \ldots,  \mathrm{r}(x_d) \big) . 
\]
Note that $\mathbf{r} (x_1, \ldots, x_d)$ is the point on the integer lattice closest to $(x_1, \ldots, x_d)$.

%

Since $\mathcal{E}( A^{-T} \Z^d )$ is an orthogonal basis for $L^2(A[0,1]^d)$ whenever $A \in \R^{d \times d}$ is a nonsingular matrix, 
a natural candidate for $\Gamma \subseteq \Z^d$ in Theorem~\ref{thm:main2-RB-integer} is the set obtained by rounding each element of $A^{-T} \Z^d$ to its closest points in $\Z^d$, that is, the set $\mathbf{r} ( A^{-T} \Z^d )$. 
The following proposition shows that such a set indeed works if $A$ is a lower triangular matrix with diagonals lying in $(0,1]$.

\begin{proposition}\label{prop:main2-lower-triangular} 
Let $H \in \R^{d \times d}$ be a lower triangular matrix with diagonals lying in $(0,1]$. 
Then there exists a set $\mathcal{C}\subseteq \Z^d$
such that $\mathcal{E}(\mathcal{C})$ is a Riesz basis for $L^2(H [0,1]^d)$.
Moreover, the set $\mathcal{C}$ can be chosen to be 
\begin{enumerate}[(i)]
\item
the set $\mathbf{r} ( H^{-T} \Z^d )$ obtained by rounding each element of $H^{-T} \, \Z^d$ to its closest point in $\Z^d$, or 

\item
the rectangular lattice $\mathrm{r} \big(\frac{1}{a_{11}} \Z + \delta_1 \big) \times \cdots \times \mathrm{r} \big(\frac{1}{a_{dd}} \Z + \delta_d \big)$, where $a_{11}, \ldots, a_{dd}$ are the diagonals of $H$, and $\delta_1, \ldots, \delta_d \in \R$ are arbitrary constants. 
\end{enumerate} 
\end{proposition}

\medskip

We are now ready to prove Theorem~\ref{thm:main2-RB-integer}.

Assume that $A [0,1]^d = R^{-1} H [0,1]^d$ where $R \in \Z^{d \times d}$ is a nonsingular integer matrix and $H \in \R^{d \times d}$ is a lower triangular matrix with diagonals lying in $(0,1]$. 
By Proposition~\ref{prop:main2-lower-triangular}, there exists a set $\mathcal{C}\subseteq \Z^d$ such that $\mathcal{E}(\mathcal{C})$ is a Riesz basis for $L^2(H [0,1]^d)$. Then $\Gamma := R^T \mathcal{C}$ is a subset of $\Z^d$ such that $\mathcal{E}(\Gamma)=\mathcal{E}(R^T \mathcal{C} )$ is a Riesz basis for $L^2((R^T)^{-T} H [0,1]^d) = L^2(R^{-1} H [0,1]^d) = L^2(A[0,1]^d)$.
    \hfill $\Box$ 

\bigskip

The rest of this section is devoted to the proof of Proposition~\ref{prop:main2-lower-triangular}.

\subsection{Proof of Proposition~\ref{prop:main2-lower-triangular} for $d = 2$} 
\label{subsec:case-d2}

We first consider the case $d = 2$. Assume that  
\begin{equation}\label{eqn:Case1-parallel-y-axis-matrix-A}
H
= 
\begin{pmatrix}
a & 0 \\
c & s
\end{pmatrix}
\quad \text{with} \quad 0 < a,s \leq 1 \;\; \text{and} \;\; c \in \R .
\end{equation}
Then
\[
H^{-1}
=
\frac{1}{as}
\left(
\begin{array}{cc}
s & 0 \\
- c & a
\end{array}
\right)
=
\left(
\begin{array}{cc}
\frac{1}{a} & 0 \\[1mm]
- \frac{c}{as} & \frac{1}{s}
\end{array}
\right)
\quad \text{and}
\quad
H^{-T}
=
\begin{pmatrix} 
\frac{1}{a} & - \frac{c}{as} \\[1.5mm]
0 & \frac{1}{s}
\end{pmatrix} 
\]
so that
\[
\begin{split}
\Lambda &= H \Z^2 = \{ m (a,c) + n (0,s) : m, n \in \Z \} = \{ (m a, mc + ns) : m, n \in \Z \} ,  \\
\Lambda^* &= H^{-T} \Z^2  = \left\{ m (\tfrac{1}{a},0) + n (- \tfrac{c}{as}, \tfrac{1}{s}) : m, n \in \Z \right\}
=  \left\{  (\tfrac{ m s - nc}{as}, \tfrac{n}{s}) : m, n \in \Z \right\} .
\end{split}
\]
Note that $\Lambda^* \subseteq \R \times \frac{1}{s} \Z$ and that for each $y \in \frac{1}{s} \Z$ the set $\{ x : (x,y) \in \Lambda^* \}$ is a shifted copy of $\frac{1}{a} \Z$.
Since $a , s \leq 1$, there is no overlapping when rounding the set $\Lambda^*$ to the nearest points in $\Z^2$,
that is, the set $\mathcal{C} = \mathbf{r} (\Lambda^*)$ has no repeated elements. 
In fact, we may write
\begin{equation}\label{eqn:mathcalC-decomposition}
\mathcal{C} = \mathbf{r} (\Lambda^*)
= \bigcup_{j \in J} \, \mathcal{X}_j  \times  \{ j \} ,
\end{equation}
where $J := \mathrm{r} \big(\frac{1}{s} \Z\big)$ and $\mathcal{X}_j := \mathrm{r} \big(\frac{1}{a} \Z + \delta_j\big)$ for some $0 \leq \delta_j < \frac{1}{a} \Z$   for each $j \in J$.

\begin{example}\label{ex:1oversqrt3-case}
Let $A= 
\begin{pmatrix} 
\frac{1}{\sqrt{3}} & 0 \\[1.5mm] 
\frac{1}{\sqrt{5}} & \frac{1}{\sqrt{2}} 
\end{pmatrix}$.
By Theorem~\ref{thm:cubes-tiling-expRB}, 
there is no set $\Gamma \subseteq \Z^2$ such that $\mathcal{E}(\Gamma)$ is an orthogonal basis for $L^2( A [0,1]^2 )$.
However, Theorem~\ref{thm:main2-RB-integer} guarantees the existence of a set $\Gamma\subseteq \Z^2$ such that $\mathcal{E}(\Gamma)$ is a Riesz basis for $L^2(A [0,1]^2)$.
Indeed, we have $A [0,1]^d = R^{-1} H [0,1]^d$
with $R= \begin{pmatrix}
    1&0\\ 
    0&1
\end{pmatrix}$ 
and
$H
= 
\begin{pmatrix} 
\frac{1}{\sqrt{3}} & 0 \\[1.5mm] 
\frac{1}{\sqrt{5}} & \frac{1}{\sqrt{2}} 
\end{pmatrix}$.
Then
$H^{-T}
=
\begin{pmatrix}
\sqrt{3} & - \frac{\sqrt{6}}{\sqrt{5}}  \\[2mm] 
0 & \sqrt{2}
\end{pmatrix}$, 
so the corresponding lattices are $\Lambda = H \Z^2 = \{ m (\frac{1}{\sqrt{3}}, \frac{1}{\sqrt{5}}) + n (0,\frac{1}{\sqrt{2}}) : m, n \in \Z \}$ and $\Lambda^* = H^{-T} \Z^2  = \{ m ( \sqrt{3} , 0 ) + n ( - \tfrac{\sqrt{6}}{\sqrt{5}} , \sqrt{2} ) : m, n \in \Z \}$.
Note that $\Lambda^* \subseteq \R \times \sqrt{2} \Z$ and
that for each $y \in \sqrt{2} \Z$ the set $\{ x : (x,y) \in \Lambda^* \}$ is a shifted copy of $\sqrt{3} \Z$.
Then 
\[
\begin{split}
J &= \mathrm{r} \big( \sqrt{2} \Z \big) = \mathrm{r} \big( \{   \ldots, \; -4.24, \; -2.83, \; -1.41, \; 0, \; 1.41, \; 2.83, \; 4.24, \; \cdots \} \big) \\
&= \{   \cdots, \; -4, \; -3, \; -1, \; 0, \; 1, \; 3, \; 4, \; \cdots \}
\end{split}
\]
and
\[
\begin{split}
&\qquad \;\; \vdots \\
\mathcal{X}_{3} &= \mathrm{r} \big( \sqrt{3} \Z - 2 \tfrac{\sqrt{6}}{\sqrt{5}} \big)
 = \mathrm{r} \big( \{   \cdots, \; -3.92, \; -2.19 \;( \approx - 2 \tfrac{\sqrt{6}}{\sqrt{5}}) , \; -0.46, \; 1.27, \; 3.01, \; \cdots \} \big) \\
 &= \{   \cdots, \; -4, \; -2, \; 0, \; 1, \; 3, \; \cdots \},  \\
\mathcal{X}_{1} &= \mathrm{r} \big( \sqrt{3} \Z - \tfrac{\sqrt{6}}{\sqrt{5}} \big)
  = \mathrm{r} \big( \{   \cdots, \; -2.83, \; -1.10 \;( \approx - \tfrac{\sqrt{6}}{\sqrt{5}}) , \; 0.64, \; 2.37, \; 4.10, \; \cdots \} \big) \\
& = \{   \cdots, \; -3, \; -1, \; 1, \; 2, \; 4, \; \cdots \},  \\
\mathcal{X}_{0} &= \mathrm{r} \big( \sqrt{3} \Z  \big)
  = \mathrm{r} \big( \{   \cdots, \; -3.46, \; -1.73, \; 0, \; 1.73, \; 3.46, \; \cdots \} \big) \\
& = \{   \cdots, \; -3, \; -2, \; 0, \; 2, \; 3, \; \cdots \} , \\
\mathcal{X}_{-1} &= \mathrm{r} \big( \sqrt{3} \Z +  \tfrac{\sqrt{6}}{\sqrt{5}} \big)
 = \mathrm{r} \big( \{   \cdots, \; -4.10, \; -2.37, \; -0.64, \; 1.10 \;( \approx \tfrac{\sqrt{6}}{\sqrt{5}}), \; 2.83, \; \cdots \} \big) \\
& = \{   \cdots, \; -4, \; -2, \; -1, \; 1, \; 3, \; \cdots \} , \\
&\qquad \;\; \vdots \quad .\quad
\end{split}
\]
Consequently, we have
\[
\begin{split}
\mathcal{C}
&= \bigcup_{j \in J} \, \mathcal{X}_j  \times  \{ j \} \\
&= \cdots \;\cup\; \{  \cdots, \; -4, \; -2, \; 0, \; 1, \; 3, \; \cdots \} \times \{3\} \\
& \qquad\quad \;\cup\;  \{  \cdots, \; -3, \; -1, \; 1, \; 2, \; 4, \; \cdots \}   \times \{1\} \\
& \qquad\quad \;\cup\;  \{   \cdots, \; -3, \; -2, \; 0, \; 2, \; 3, \; \cdots \}   \times \{0\} \\
& \qquad\quad \;\cup\;  \{   \cdots, \; -4, \; -2, \; -1, \; 1, \; 3, \; \cdots \}  \times \{-1\} \;\cup\;  \cdots .
\end{split}
\] 
The lattices $\Lambda = H \mathbb{Z}^2$ and $\Lambda^* = H^{-T} \mathbb{Z}^2$ are depicted in Figure~\ref{fig:lattice-dual-lattice}, while the set $\mathcal{C}$ is illustrated in Figure~\ref{fig:mathcalC}.
\end{example}

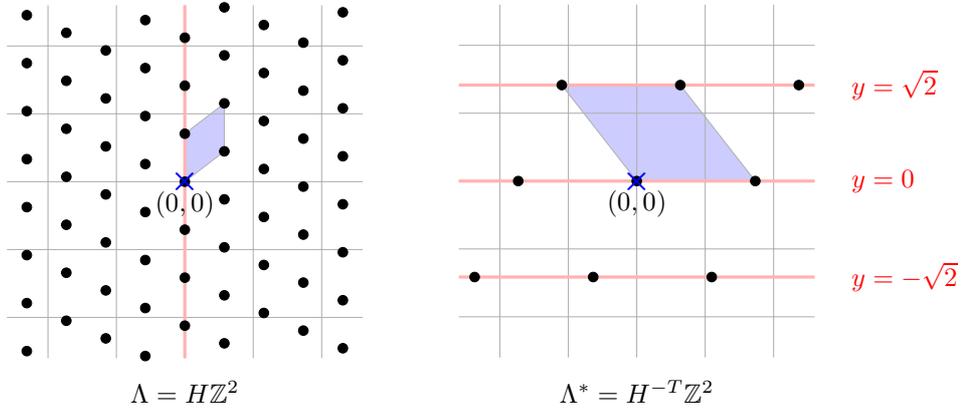
\begin{figure}[h!]
\begin{center}
\begin{tikzpicture}[scale=0.9]
\draw[black!30] (0,0) -- ({1/sqrt(3)}, {1/sqrt(5)}) -- ({1/sqrt(3)}, {1/sqrt(5) + 1/sqrt(2)}) -- (0, {1/sqrt(2)}) -- (0,0) [fill=blue!20];
    \draw[black!30] (-2.6,-2.6) grid (2.6,2.6);
	\draw[very thick,red!30] (0,-2.6) -- (0,2.6);

                \draw ({-4/sqrt(3)},{-4/sqrt(5) + -1/sqrt(2)}) circle (0.7mm) [fill=black];
                \draw ({-4/sqrt(3)},{-4/sqrt(5) + 0/sqrt(2)}) circle (0.7mm) [fill=black];
                \draw ({-4/sqrt(3)},{-4/sqrt(5) + 1/sqrt(2)}) circle (0.7mm) [fill=black];
                \draw ({-4/sqrt(3)},{-4/sqrt(5) + 2/sqrt(2)}) circle (0.7mm) [fill=black];
                \draw ({-4/sqrt(3)},{-4/sqrt(5) + 3/sqrt(2)}) circle (0.7mm) [fill=black];
                \draw ({-4/sqrt(3)},{-4/sqrt(5) + 4/sqrt(2)}) circle (0.7mm) [fill=black];
                \draw ({-4/sqrt(3)},{-4/sqrt(5) + 5/sqrt(2)}) circle (0.7mm) [fill=black];
                \draw ({-4/sqrt(3)},{-4/sqrt(5) + 6/sqrt(2)}) circle (0.7mm) [fill=black];

                \draw ({-3/sqrt(3)},{-3/sqrt(5) + -1/sqrt(2)}) circle (0.7mm) [fill=black];
                \draw ({-3/sqrt(3)},{-3/sqrt(5) + 0/sqrt(2)}) circle (0.7mm) [fill=black];
                \draw ({-3/sqrt(3)},{-3/sqrt(5) + 1/sqrt(2)}) circle (0.7mm) [fill=black];
                \draw ({-3/sqrt(3)},{-3/sqrt(5) + 2/sqrt(2)}) circle (0.7mm) [fill=black];
                \draw ({-3/sqrt(3)},{-3/sqrt(5) + 3/sqrt(2)}) circle (0.7mm) [fill=black];
                \draw ({-3/sqrt(3)},{-3/sqrt(5) + 4/sqrt(2)}) circle (0.7mm) [fill=black];
                \draw ({-3/sqrt(3)},{-3/sqrt(5) + 5/sqrt(2)}) circle (0.7mm) [fill=black];

                \draw ({-2/sqrt(3)},{-2/sqrt(5) + -2/sqrt(2)}) circle (0.7mm) [fill=black];
                \draw ({-2/sqrt(3)},{-2/sqrt(5) + -1/sqrt(2)}) circle (0.7mm) [fill=black];
                \draw ({-2/sqrt(3)},{-2/sqrt(5) + 0/sqrt(2)}) circle (0.7mm) [fill=black];
                \draw ({-2/sqrt(3)},{-2/sqrt(5) + 1/sqrt(2)}) circle (0.7mm) [fill=black];
                \draw ({-2/sqrt(3)},{-2/sqrt(5) + 2/sqrt(2)}) circle (0.7mm) [fill=black];
                \draw ({-2/sqrt(3)},{-2/sqrt(5) + 3/sqrt(2)}) circle (0.7mm) [fill=black];
                \draw ({-2/sqrt(3)},{-2/sqrt(5) + 4/sqrt(2)}) circle (0.7mm) [fill=black];

                \draw ({-1/sqrt(3)},{-1/sqrt(5) + -3/sqrt(2)}) circle (0.7mm) [fill=black];
                \draw ({-1/sqrt(3)},{-1/sqrt(5) + -2/sqrt(2)}) circle (0.7mm) [fill=black];
                \draw ({-1/sqrt(3)},{-1/sqrt(5) + -1/sqrt(2)}) circle (0.7mm) [fill=black];
                \draw ({-1/sqrt(3)},{-1/sqrt(5) + 0/sqrt(2)}) circle (0.7mm) [fill=black];
                \draw ({-1/sqrt(3)},{-1/sqrt(5) + 1/sqrt(2)}) circle (0.7mm) [fill=black];
                \draw ({-1/sqrt(3)},{-1/sqrt(5) + 2/sqrt(2)}) circle (0.7mm) [fill=black];
                \draw ({-1/sqrt(3)},{-1/sqrt(5) + 3/sqrt(2)}) circle (0.7mm) [fill=black];
                \draw ({-1/sqrt(3)},{-1/sqrt(5) + 4/sqrt(2)}) circle (0.7mm) [fill=black];

                \draw ({-0/sqrt(3)},{-0/sqrt(5) + -3/sqrt(2)}) circle (0.7mm) [fill=black];
                \draw ({-0/sqrt(3)},{-0/sqrt(5) + -2/sqrt(2)}) circle (0.7mm) [fill=black];
                \draw ({-0/sqrt(3)},{-0/sqrt(5) + -1/sqrt(2)}) circle (0.7mm) [fill=black];
                \draw ({-0/sqrt(3)},{-0/sqrt(5) + -0/sqrt(2)}) circle (0.7mm) [fill=black];
                \draw ({-0/sqrt(3)},{-0/sqrt(5) + 1/sqrt(2)}) circle (0.7mm) [fill=black];
                \draw ({-0/sqrt(3)},{-0/sqrt(5) + 2/sqrt(2)}) circle (0.7mm) [fill=black];
                \draw ({-0/sqrt(3)},{-0/sqrt(5) + 3/sqrt(2)}) circle (0.7mm) [fill=black];

                \draw ({1/sqrt(3)},{1/sqrt(5) + -4/sqrt(2)}) circle (0.7mm) [fill=black];
                \draw ({1/sqrt(3)},{1/sqrt(5) + -3/sqrt(2)}) circle (0.7mm) [fill=black];
                \draw ({1/sqrt(3)},{1/sqrt(5) + -2/sqrt(2)}) circle (0.7mm) [fill=black];
                \draw ({1/sqrt(3)},{1/sqrt(5) + -1/sqrt(2)}) circle (0.7mm) [fill=black];
                \draw ({1/sqrt(3)},{1/sqrt(5) + 0/sqrt(2)}) circle (0.7mm) [fill=black];
                \draw ({1/sqrt(3)},{1/sqrt(5) + 1/sqrt(2)}) circle (0.7mm) [fill=black];
                \draw ({1/sqrt(3)},{1/sqrt(5) + 2/sqrt(2)}) circle (0.7mm) [fill=black];
                \draw ({1/sqrt(3)},{1/sqrt(5) + 3/sqrt(2)}) circle (0.7mm) [fill=black];

                \draw ({2/sqrt(3)},{2/sqrt(5) + -4/sqrt(2)}) circle (0.7mm) [fill=black];
                \draw ({2/sqrt(3)},{2/sqrt(5) + -3/sqrt(2)}) circle (0.7mm) [fill=black];
                \draw ({2/sqrt(3)},{2/sqrt(5) + -2/sqrt(2)}) circle (0.7mm) [fill=black];
                \draw ({2/sqrt(3)},{2/sqrt(5) + -1/sqrt(2)}) circle (0.7mm) [fill=black];
                \draw ({2/sqrt(3)},{2/sqrt(5) + 0/sqrt(2)}) circle (0.7mm) [fill=black];
                \draw ({2/sqrt(3)},{2/sqrt(5) + 1/sqrt(2)}) circle (0.7mm) [fill=black];
                \draw ({2/sqrt(3)},{2/sqrt(5) + 2/sqrt(2)}) circle (0.7mm) [fill=black];

                \draw ({3/sqrt(3)},{3/sqrt(5) + -5/sqrt(2)}) circle (0.7mm) [fill=black];
                \draw ({3/sqrt(3)},{3/sqrt(5) + -4/sqrt(2)}) circle (0.7mm) [fill=black];
                \draw ({3/sqrt(3)},{3/sqrt(5) + -3/sqrt(2)}) circle (0.7mm) [fill=black];
                \draw ({3/sqrt(3)},{3/sqrt(5) + -2/sqrt(2)}) circle (0.7mm) [fill=black];
                \draw ({3/sqrt(3)},{3/sqrt(5) + -1/sqrt(2)}) circle (0.7mm) [fill=black];
                \draw ({3/sqrt(3)},{3/sqrt(5) + 0/sqrt(2)}) circle (0.7mm) [fill=black];
                \draw ({3/sqrt(3)},{3/sqrt(5) + 1/sqrt(2)}) circle (0.7mm) [fill=black];

                \draw ({4/sqrt(3)},{4/sqrt(5) + -6/sqrt(2)}) circle (0.7mm) [fill=black];
                \draw ({4/sqrt(3)},{4/sqrt(5) + -5/sqrt(2)}) circle (0.7mm) [fill=black];
                \draw ({4/sqrt(3)},{4/sqrt(5) + -4/sqrt(2)}) circle (0.7mm) [fill=black];
                \draw ({4/sqrt(3)},{4/sqrt(5) + -3/sqrt(2)}) circle (0.7mm) [fill=black];
                \draw ({4/sqrt(3)},{4/sqrt(5) + -2/sqrt(2)}) circle (0.7mm) [fill=black];
                \draw ({4/sqrt(3)},{4/sqrt(5) + -1/sqrt(2)}) circle (0.7mm) [fill=black];
                \draw ({4/sqrt(3)},{4/sqrt(5) + 0/sqrt(2)}) circle (0.7mm) [fill=black];
                \draw ({4/sqrt(3)},{4/sqrt(5) + 1/sqrt(2)}) circle (0.7mm) [fill=black];

\node[blue] at (0,0) {\large $\boldsymbol{\times}$};
\node[below] at (0,0) {$(0,0)$};
\node[below] at (0,-2.8) {$\Lambda = H \Z^2$};
\end{tikzpicture}
\hspace{10mm}
\begin{tikzpicture}[scale=0.9]
\draw[black!30] (0,0) -- ({1*sqrt(3)}, 0) -- ({1*sqrt(3) - 1*sqrt(6)/sqrt(5)}, {1*sqrt(2)}) -- ({ - 1*sqrt(6)/sqrt(5)}, {1*sqrt(2)}) -- (0,0) [fill=blue!20];
    \draw[black!30] (-2.6,-2.6) grid (2.6,2.6);
	\draw[very thick,red!30] (-2.6,{1*sqrt(2)}) -- (2.6,{1*sqrt(2)});
	\draw[very thick,red!30] (-2.6,0) -- (2.6,0);
	\draw[very thick,red!30] (-2.6,{-1*sqrt(2)}) -- (2.6,{-1*sqrt(2)});
                \draw ({0*sqrt(3) - 1*sqrt(6)/sqrt(5)}, {1*sqrt(2)}) circle (0.7mm) [fill=black];
                \draw ({1*sqrt(3) - 1*sqrt(6)/sqrt(5)}, {1*sqrt(2)}) circle (0.7mm) [fill=black];
                \draw ({2*sqrt(3) - 1*sqrt(6)/sqrt(5)}, {1*sqrt(2)}) circle (0.7mm) [fill=black];

                \draw ({-1*sqrt(3) - 0*sqrt(6)/sqrt(5)}, {0*sqrt(2)}) circle (0.7mm) [fill=black];
                \draw ({0*sqrt(3) - 0*sqrt(6)/sqrt(5)}, {0*sqrt(2)}) circle (0.7mm) [fill=black];
                \draw ({1*sqrt(3) - 0*sqrt(6)/sqrt(5)}, {0*sqrt(2)}) circle (0.7mm) [fill=black];

                \draw ({-2*sqrt(3) +1*sqrt(6)/sqrt(5)}, {-1*sqrt(2)}) circle (0.7mm) [fill=black];
                \draw ({-1*sqrt(3) +1*sqrt(6)/sqrt(5)}, {-1*sqrt(2)}) circle (0.7mm) [fill=black];
                \draw ({0*sqrt(3) +1*sqrt(6)/sqrt(5)}, {-1*sqrt(2)}) circle (0.7mm) [fill=black];

\node[blue] at (0,0) {\large $\boldsymbol{\times}$};
\node[below] at (0,0) {$(0,0)$};
\node[right] at (3,{1*sqrt(2)}) {$\red{y = \sqrt{2}}$};
\node[right] at (3,{0*sqrt(2)}) {$\red{y = 0}$};
\node[right] at (3,{-1*sqrt(2)}) {$\red{y = - \sqrt{2}}$};
\node[below] at (0,-2.8) {$\Lambda^* = H^{-T} \Z^2$};
\end{tikzpicture}
\end{center}
\caption{The lattice $\Lambda = H \Z^2$ and its dual lattice $\Lambda^* = H^{-T} \Z^d$}
\label{fig:lattice-dual-lattice}
\end{figure}

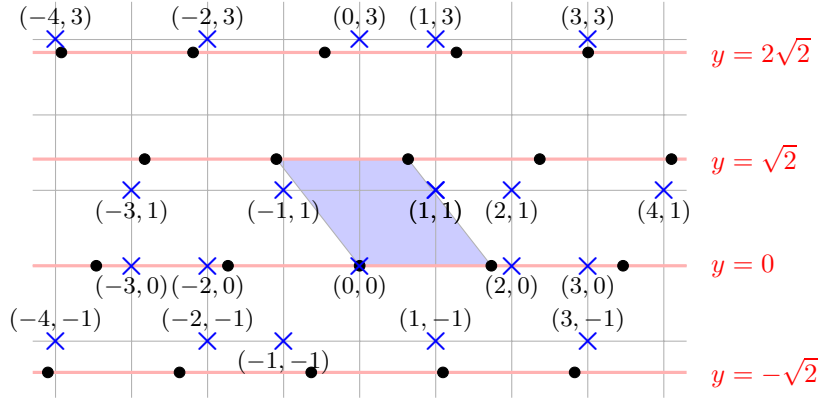
\begin{figure}[h!]
\centering
\begin{tikzpicture}[scale=1]
\draw[black!30] (0,0) -- ({1*sqrt(3)}, 0) -- ({1*sqrt(3) - 1*sqrt(6)/sqrt(5)}, {1*sqrt(2)}) -- ({ - 1*sqrt(6)/sqrt(5)}, {1*sqrt(2)}) -- (0,0) [fill=blue!20];
    \draw[black!30] (-4.3,-1.75) grid (4.3,3.5);
	\draw[very thick,red!30] (-4.3,{2*sqrt(2)}) -- (4.3,{2*sqrt(2)});
	\draw[very thick,red!30] (-4.3,{1*sqrt(2)}) -- (4.3,{1*sqrt(2)});
	\draw[very thick,red!30] (-4.3,0) -- (4.3,0);
	\draw[very thick,red!30] (-4.3,{-1*sqrt(2)}) -- (4.3,{-1*sqrt(2)});

                \draw ({-1*sqrt(3) - 2*sqrt(6)/sqrt(5)}, {2*sqrt(2)}) circle (0.7mm) [fill=black];
                \draw ({0*sqrt(3) - 2*sqrt(6)/sqrt(5)}, {2*sqrt(2)}) circle (0.7mm) [fill=black];
                \draw ({1*sqrt(3) - 2*sqrt(6)/sqrt(5)}, {2*sqrt(2)}) circle (0.7mm) [fill=black];
                \draw ({2*sqrt(3) - 2*sqrt(6)/sqrt(5)}, {2*sqrt(2)}) circle (0.7mm) [fill=black];
                \draw ({3*sqrt(3) - 2*sqrt(6)/sqrt(5)}, {2*sqrt(2)}) circle (0.7mm) [fill=black];

                \draw ({-1*sqrt(3) - 1*sqrt(6)/sqrt(5)}, {1*sqrt(2)}) circle (0.7mm) [fill=black];
                \draw ({0*sqrt(3) - 1*sqrt(6)/sqrt(5)}, {1*sqrt(2)}) circle (0.7mm) [fill=black];
                \draw ({1*sqrt(3) - 1*sqrt(6)/sqrt(5)}, {1*sqrt(2)}) circle (0.7mm) [fill=black];
                \draw ({2*sqrt(3) - 1*sqrt(6)/sqrt(5)}, {1*sqrt(2)}) circle (0.7mm) [fill=black];
                \draw ({3*sqrt(3) - 1*sqrt(6)/sqrt(5)}, {1*sqrt(2)}) circle (0.7mm) [fill=black];

                \draw ({-2*sqrt(3) - 0*sqrt(6)/sqrt(5)}, {0*sqrt(2)}) circle (0.7mm) [fill=black];
                \draw ({-1*sqrt(3) - 0*sqrt(6)/sqrt(5)}, {0*sqrt(2)}) circle (0.7mm) [fill=black];
                \draw ({0*sqrt(3) - 0*sqrt(6)/sqrt(5)}, {0*sqrt(2)}) circle (0.7mm) [fill=black];
                \draw ({1*sqrt(3) - 0*sqrt(6)/sqrt(5)}, {0*sqrt(2)}) circle (0.7mm) [fill=black];
                \draw ({2*sqrt(3) - 0*sqrt(6)/sqrt(5)}, {0*sqrt(2)}) circle (0.7mm) [fill=black];

                \draw ({-3*sqrt(3) +1*sqrt(6)/sqrt(5)}, {-1*sqrt(2)}) circle (0.7mm) [fill=black];
                \draw ({-2*sqrt(3) +1*sqrt(6)/sqrt(5)}, {-1*sqrt(2)}) circle (0.7mm) [fill=black];
                \draw ({-1*sqrt(3) +1*sqrt(6)/sqrt(5)}, {-1*sqrt(2)}) circle (0.7mm) [fill=black];
                \draw ({0*sqrt(3) +1*sqrt(6)/sqrt(5)}, {-1*sqrt(2)}) circle (0.7mm) [fill=black];
                \draw ({1*sqrt(3) +1*sqrt(6)/sqrt(5)}, {-1*sqrt(2)}) circle (0.7mm) [fill=black];




\node[blue] at (-4,3) {\large $\boldsymbol{\times}$};
\node[above] at (-4,3) {\scriptsize $(-4,3)$};
\node[blue] at (-2,3) {\large $\boldsymbol{\times}$};
\node[above] at (-2,3) {\scriptsize $(-2,3)$};
\node[blue] at (0,3) {\large $\boldsymbol{\times}$};
\node[above] at (0,3) {\scriptsize $(0,3)$};
\node[blue] at (1,3) {\large $\boldsymbol{\times}$};
\node[above] at (1,3) {\scriptsize $(1,3)$};
\node[blue] at (3,3) {\large $\boldsymbol{\times}$};
\node[above] at (3,3) {\scriptsize $(3,3)$};

\node[blue] at (-3,1) {\large $\boldsymbol{\times}$};
\node[below] at (-3,1) {\scriptsize $(-3,1)$};
\node[blue] at (-1,1) {\large $\boldsymbol{\times}$};
\node[below] at (-1,1) {\scriptsize $(-1,1)$};
\node[blue] at (1,1) {\large $\boldsymbol{\times}$};
\node[below] at (1,1) {\scriptsize $(1,1)$};
\node[blue] at (1,1) {\large $\boldsymbol{\times}$};
\node[below] at (2,1) {\scriptsize $(2,1)$};
\node[blue] at (2,1) {\large $\boldsymbol{\times}$};
\node[below] at (1,1) {\scriptsize $(1,1)$};
\node[blue] at (4,1) {\large $\boldsymbol{\times}$};
\node[below] at (4,1) {\scriptsize $(4,1)$};

\node[blue] at (-3,0) {\large $\boldsymbol{\times}$};
\node[below] at (-3,0) {\scriptsize $(-3,0)$};
\node[blue] at (-2,0) {\large $\boldsymbol{\times}$};
\node[below] at (-2,0) {\scriptsize $(-2,0)$};
\node[blue] at (0,0) {\large $\boldsymbol{\times}$};
\node[below] at (0,0) {\scriptsize $(0,0)$};
\node[blue] at (2,0) {\large $\boldsymbol{\times}$};
\node[below] at (2,0) {\scriptsize $(2,0)$};
\node[blue] at (3,0) {\large $\boldsymbol{\times}$};
\node[below] at (3,0) {\scriptsize $(3,0)$};

\node[blue] at (-4,-1) {\large $\boldsymbol{\times}$};
\node[above] at (-4,-1) {\scriptsize $(-4,-1)$};
\node[blue] at (-2,-1) {\large $\boldsymbol{\times}$};
\node[above] at (-2,-1) {\scriptsize $(-2,-1)$};
\node[blue] at (-1,-1) {\large $\boldsymbol{\times}$};
\node[below] at (-1,-1) {\scriptsize $(-1,-1)$};
\node[blue] at (1,-1) {\large $\boldsymbol{\times}$};
\node[above] at (1,-1) {\scriptsize $(1,-1)$};
\node[blue] at (3,-1) {\large $\boldsymbol{\times}$};
\node[above] at (3,-1) {\scriptsize $(3,-1)$};



\node[right] at (4.5,{2*sqrt(2)}) {$\red{y = 2\sqrt{2}}$};
\node[right] at (4.5,{1*sqrt(2)}) {$\red{y = \sqrt{2}}$};
\node[right] at (4.5,{0*sqrt(2)}) {$\red{y = 0}$};
\node[right] at (4.5,{-1*sqrt(2)}) {$\red{y = - \sqrt{2}}$};
\end{tikzpicture}
\caption{The lattice $\Lambda^*$ (black dots) and its rounded set $\mathcal{C} = \mathbf{r} (\Lambda^*)$ (blue crosses)}
\label{fig:mathcalC}
\end{figure}

We will need the following lemma. 

\begin{lemma}[\cite{PRW21}]
\label{lem:a-irrational-Beatty-Fraenkel}
Let $0 < \alpha \leq 1$.
For every $\beta \in \R$, the system $\mathcal{E}( \lfloor \frac{\Z+\beta}{\alpha} \rfloor )$ is a Riesz basis for $L^2[0,\alpha]$ with lower Riesz bound 
depending on $\alpha$ but not on $\beta$.
\end{lemma}

Sets of the form $\lfloor \frac{\Z+\beta}{\alpha} \rfloor$ with irrational $\alpha \in \R$ are known as \emph{Beatty--Fraenkel sequences} \cite{Be26,Be27,Fr69}. 
For self-containedness of the paper, we provide a proof of Lemma \ref{lem:a-irrational-Beatty-Fraenkel} in Appendix~\ref{sec:proof-lem:a-irrational-Beatty-Fraenkel}.

By definition of $J$ and $\mathcal{X}_j$, we obtain the following from Lemma \ref{lem:a-irrational-Beatty-Fraenkel}.
\begin{enumerate}[(I)]
\item
The system $\mathcal{E}(J)$ is a Riesz basis for $L^2(I)$ with any interval $I \subseteq \R$ of length $s$, and with Riesz bounds depending only on $s$.
This means that $J$ is a set of uniqueness, sampling and interpolation for $PW(I)$ with any interval $I \subseteq \R$ of length $s$.

\item
For each $j \in J$, the system $\mathcal{E}(\mathcal{X}_j)$ is a Riesz basis for $L^2(I)$ with any interval $I \subseteq \R$ of length $a$, 
and with Riesz bounds depending only on $a$.
This means that for each $j \in J$, the set $\mathcal{X}_j$ is a set of uniqueness, sampling and interpolation for $PW(I)$ with any interval $I \subseteq \R$ of length $a$.
\end{enumerate}

It should be noted that the set $\mathcal{X}_j$ depends on $j \in J$. 
However, this dependency can be removed by defining $\mathcal{X}_j = \mathrm{r} \big(\frac{1}{a} \Z\big)$ for all $j \in J$.  
This yields a tensor product set $\mathcal{C} = \mathrm{r} \big(\frac{1}{a} \Z \big) \times \mathrm{r} \big(\frac{1}{s} \Z\big)$. 
As we shall see in the proof below, the properties (I) and (II) are all that we need from the set $\mathcal{C}$; hence, such a tensor product set also works.

To establish Proposition~\ref{prop:main2-lower-triangular} for $d=2$, it is enough to show the following.

\begin{proposition}\label{prop:base-case-d2}
Let $H$ be a matrix of the form \eqref{eqn:Case1-parallel-y-axis-matrix-A}. 
Then there exists a set $\mathcal{C}\subseteq \Z^2$
such that $\mathcal{E}(\mathcal{C})$ is a Riesz basis for $L^2(H [0,1]^2)$.
Moreover, we may choose 
\begin{enumerate}[(i)]
\item
$\mathcal{C} = \mathbf{r} ( H^{-T}\mathbb{Z}^{d} )$ which can be expressed as \eqref{eqn:mathcalC-decomposition}, or 

\item
$\mathcal{C} = \mathrm{r} \big(\frac{1}{a} \Z \big) \times \mathrm{r} \big(\frac{1}{s} \Z\big)$ which can be expressed as
\begin{equation}\label{eqn:mathcalC-decomposition-tensor}
\bigcup_{j \in J} \, \mathcal{X}_j  \times  \{ j \} 
\quad \text{with} \;\; J = \mathrm{r} \big(\tfrac{1}{s} \Z\big) \;\; \text{and} \;\; \mathcal{X}_{\boldsymbol{j}} = \mathrm{r} \big(\tfrac{1}{a} \Z \big) \;\; \text{for all} \;\; j \in J .
\end{equation}
\end{enumerate} 
\end{proposition}

\begin{remark} 
\rm
The set $\mathrm{r} \big(\frac{1}{a} \Z \big) \times \mathrm{r} \big(\frac{1}{s} \Z\big)$ in ({\romannumeral 2}) can be easily replaced with $\mathrm{r} \big(\frac{1}{a} \Z + \delta_1  \big) \times \mathrm{r} \big(\frac{1}{s} \Z + \delta_2 \big)$ for arbitrary constants $\delta_1,   \delta_2 \in \R$ (see Lemma~\ref{lem:a-irrational-Beatty-Fraenkel}). 
\end{remark}

\begin{proof}
We will use the expressions \eqref{eqn:mathcalC-decomposition} and \eqref{eqn:mathcalC-decomposition-tensor} for $\mathcal{C}$. 
Using \eqref{eqn:Case1-parallel-y-axis-matrix-A}, we may write $S := H [0,1]^2 = \{ (\omega_1, \omega_2) : 0 \leq \omega_1 \leq a , \; \frac{c}{a}\omega_1 \leq \omega_2 \leq \frac{c}{a}\omega_1 + s \}$, see Figure~\ref{fig:setS-using-H} below.

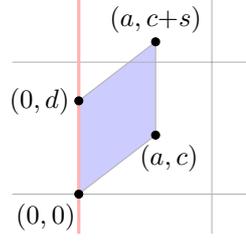
\begin{figure}[h!]
\centering
\begin{tikzpicture}[scale=1.75]
\draw[black!30] (0,0) -- ({1/sqrt(3)}, {1/sqrt(5)}) -- ({1/sqrt(3)}, {1/sqrt(5) + 1/sqrt(2)}) -- (0, {1/sqrt(2)}) -- (0,0) [fill=blue!20];
    \draw[black!30] (-0.5,-0.3) grid (1.3,1.5);
	\draw[very thick,red!30] (0,-0.3) -- (0,1.5);

                \draw ({-0/sqrt(3)},{-0/sqrt(5) + -0/sqrt(2)}) circle (0.3mm) [fill=black];
                \draw ({-0/sqrt(3)},{-0/sqrt(5) + 1/sqrt(2)}) circle (0.3mm) [fill=black];

                \draw ({1/sqrt(3)},{1/sqrt(5) + 0/sqrt(2)}) circle (0.3mm) [fill=black];
                \draw ({1/sqrt(3)},{1/sqrt(5) + 1/sqrt(2)}) circle (0.3mm) [fill=black];

\node[below] at (-0.25,0) {$(0,0)$};
\node[left] at ({-0/sqrt(3)},{-0/sqrt(5) + 1/sqrt(2)}) {$(0,d)$};
\node[below] at ({1/sqrt(3) + 0.1},{1/sqrt(5) + 0/sqrt(2)}) {$(a,c)$};
\node[above] at ({1/sqrt(3)},{1/sqrt(5) + 1/sqrt(2)}) {$(a,c{+}s)$};
\end{tikzpicture}
\caption{The set $S = H [0,1]^2$}
\label{fig:setS-using-H}
\end{figure}

\medskip

\noindent
\textbf{Completeness}. \
To show that $\mathcal{E}(\mathcal{C})$ is complete in $L^2(S)$, we suppose that $f \in PW(S)$ satisfies $f(\gamma) = 0$ for all $\gamma \in \mathcal{C}$ and show that $f = 0$.

Note that since $f \in PW(S)$, its $2$-dimensional Fourier transform $F := \widehat{f} \in L^2(\R^2)$ is supported in $S \subseteq \R^2$.
The set $S$ is bounded, so the function $f$ is analytic (in particular, continuous)
and the following equation\footnote{This equation is also known as the \emph{projection-slice theorem} or \emph{Fourier slice theorem} \cite{Br56}.} holds for all $x,y \in \R$:
\[
\begin{split}
f(x,y)
&= \iint F(\omega_1,\omega_2) \, e^{2 \pi i (x \omega_1 + y \omega_2)} \, d\omega_1 \, d\omega_2 \\
&= \int \underbrace{\left( \int F(\omega_1,\omega_2) \, e^{2 \pi i y \omega_2} \, d\omega_2 \right) }_{=: P_y F (\omega_1)} \, e^{2 \pi i x \omega_1} \, d\omega_1 \\
&= \mathcal{F}^{-1} [P_y F] (x) ,
\end{split}
\]
where $\mathcal{F} := \mathcal{F}_1$ denotes the $1$-dimensional Fourier transform and
$F \mapsto P_y F$ is the linear mapping defined by $P_y F (\omega_1) := \int F(\omega_1,\omega_2) \, e^{2 \pi i y \omega_2} \, d\omega_2$.
Since $\supp F \subseteq S \subseteq [0,a] \times \R$, we have $\supp P_y F \subseteq [0,a]$, so that $\mathcal{F}^{-1} [P_y F] \in PW[0,a]$.
Recall that for each $j \in J$, the set $\mathcal{X}_j$ is a set of uniqueness for $PW[0,a]$.
Since $f(\cdot,j) = \mathcal{F}^{-1} [P_j F] \in PW[0,a]$ and since $f(x,j) = 0$ for all $x \in \mathcal{X}_j$, we deduce that $f(\cdot,j) = 0$, equivalently, $P_j F = 0$.
Therefore, we have $P_j F = 0$ for all $j \in J$.

Now, it holds for all $j \in J$ and $\omega_1 \in \R$ that
\[
0 = P_j F (\omega_1)
:= \int F(\omega_1,\omega_2) \, e^{2 \pi i j \omega_2} \, d\omega_2
= \mathcal{F}^{-1} [F(\omega_1,\cdot)] (j) .
\]
Fix any $\omega_1 \in [0,a]$. Then
$\supp F(\omega_1,\cdot) \subseteq [\frac{c}{a}\omega_1,\frac{c}{a}\omega_1 {+} s]$ by the geometry of $S$, which implies that $\mathcal{F}^{-1} [F(\omega_1,\cdot)] \in PW[\frac{c}{a}\omega_1,\frac{c}{a}\omega_1 {+} s]$.
Since $J \subseteq \Z$ is a set of uniqueness for $PW(I)$ with any interval $I \subseteq \R$ of length $s$, we deduce that $\mathcal{F}^{-1} [F(\omega_1,\cdot)] = 0$ or equivalently, $F(\omega_1,\cdot) = 0$.
That is, $F(\omega_1,\cdot) = 0$ for all $\omega_1 \in [0,a]$.
Since $\supp F \subseteq S \subseteq [0,a] \times \R$, we conclude that $F = 0$ and in turn, $f = 0$.

\medskip

\noindent
\textbf{Riesz sequence}. \
Since $\mathcal{C} \subseteq \Z {\times} \Z$ is separated and since $S$ is bounded, the system $\mathcal{E}(\mathcal{C})$ is a Bessel sequence in $L^2(S)$. Thus, to prove that $\mathcal{E}(\mathcal{C})$ is a Riesz sequence in $L^2(S)$, it is enough to show that $\mathcal{C}$ is a set of interpolation for $PW(S)$.
To show this, we will fix an arbitrary sequence $\{b_\gamma\}_{\gamma \in \mathcal{C}} \in \ell_2 (\mathcal{C})$, and construct a function $f \in PW(S)$ satisfying $f(\gamma) = b_{\gamma}$ for all $\gamma \in \mathcal{C}$.

First, recall that for each $j \in J$ the set $\mathcal{X}_j$ is a set of interpolation for $PW[0,a]$.
Thus, for each $j \in J$ there exists a function $Q_j \in L^2[0,a]$ (correspondingly, $\mathcal{F}^{-1} [Q_j] \in PW [0,a]$) such that $\mathcal{F}^{-1} [Q_j] (x) = b_{(x,j)}$ for all $x \in \mathcal{X}_j$.

Now we claim that for a.e.~$\omega_1 \in [0,a]$, the sequence $\{ Q_j (\omega_1) \}_{j \in J}$ belongs in $\ell_2 (J)$.
To see this, we use the fact that for each $j \in J$ the set $\mathcal{X}_j \subseteq \R$ is a set of sampling for $PW[0,a]$, to get
\[
\sum_{x \in \mathcal{X}_j} |b_{(x,j)}|^2
= \sum_{x \in \mathcal{X}_j} |\mathcal{F}^{-1} [Q_j] (x)|^2
\;\asymp\;
\| \mathcal{F}^{-1} [Q_j] \|_{L^2(\R)}^2
= \| Q_j \|_{L^2[0,a]}^2 .
\]
That is, there are constants $0 < A_j \leq B_j < \infty$ depending only on $\mathcal{X}_j$ and $a$ such that
\begin{equation}\label{eqn:Xj-is-sampling-for-PW}
A_j \, \| Q_j \|_{L^2[0,a]}^2
\;\leq\;
\sum_{x \in \mathcal{X}_j} |b_{(x,j)}|^2
\;\leq\;
B_j \, \| Q_j \|_{L^2[0,a]}^2 .
\end{equation}
In fact, the Bessel bound $B_j$ depends only on the separation of $\mathcal{X}_j \subseteq \Z$ and the diameter of $[0,a]$ (see e.g., \cite[Theorem 17 on p.\,82, and Theorem 4 on p.\,135]{Yo01}), and therefore
$B := \sup_{j \in J} B_j < \infty$.
Moreover, we have $A := \inf_{j \in J} A_j > 0$ by the construction of $\mathcal{X}_j$ and by Lemma \ref{lem:a-irrational-Beatty-Fraenkel}.
Summing up the left inequality of \eqref{eqn:Xj-is-sampling-for-PW} for $j \in J$ gives
\begin{equation}\label{eqn:int-sumJ-Hj-omega1-is-finite}
\begin{split}
\infty
&>
\sum_{(x,j) \in \mathcal{C}} |b_{(x,j)}|^2
=\sum_{j \in J} \sum_{x \in \mathcal{X}_j} |b_{(x,j)}|^2 \\
&\geq
A \, \sum_{j \in J} \| Q_j \|_{L^2[0,a]}^2
= A \, \sum_{j \in J} \int_0^{a} |Q_j (\omega_1)|^2 \, d\omega_1
= A \, \int_0^{a} \sum_{j \in J} |Q_j (\omega_1)|^2 \, d\omega_1 ,
\end{split}
\end{equation}
which implies that $\sum_{j \in J} |Q_j (\omega_1)|^2 < \infty$ for a.e.~$\omega_1 \in [0,a]$, as desired.

Let us fix any $\omega_1 \in [0,a]$ satisfying $\{ Q_j (\omega_1) \}_{j \in J} \in \ell_2 (J)$.
Recall that $J \subseteq \Z$ is a set of sampling and interpolation for $PW(I)$ with any interval $I \subseteq \R$ of length $d$,
where the corresponding frame/Riesz bounds\footnote{A Riesz basis is also a frame with frame bounds coinciding with its Riesz bounds, see e.g., \cite[Theorem 5.4.1]{Ch16}.}, say $0 < A' \leq B' < \infty$, are \emph{independent} of $I$ by Lemma \ref{lem:a-irrational-Beatty-Fraenkel}.
The `interpolation' part implies the existence of a function $F_{\omega_1} \in L^2[\frac{c}{a}\omega_1,\frac{c}{a}\omega_1 {+} s]$ (correspondingly, $\mathcal{F}^{-1} [F_{\omega_1}] \in PW[\frac{c}{a}\omega_1,\frac{c}{a}\omega_1 {+} s]$) such that $\mathcal{F}^{-1} [F_{\omega_1}] (j) = Q_j (\omega_1)$ for all $j \in J$.
In turn, the `sampling' part implies
\[
\sum_{j \in J} |Q_j (\omega_1)|^2
= \sum_{j \in J} \big| \mathcal{F}^{-1} [F_{\omega_1}] (j) \big|^2
\;\asymp\;
\| \mathcal{F}^{-1} [F_{\omega_1}] \|_{L^2(\R)}^2
= \| F_{\omega_1} \|_{L^2[\frac{c}{a}\omega_1,\frac{c}{a}\omega_1 + s]}^2 ,
\]
that is,
\begin{equation}\label{eqn:J-is-sampling-set-with-universal-constants}
A' \, \| F_{\omega_1} \|_{L^2[\frac{c}{a}\omega_1,\frac{c}{a}\omega_1 + s]}^2
\;\leq\;
\sum_{j \in J} |Q_j (\omega_1)|^2
\;\leq\;
B' \, \| F_{\omega_1} \|_{L^2[\frac{c}{a}\omega_1,\frac{c}{a}\omega_1 + s]}^2 .
\end{equation}

Now, we define $F \in L^2(\R^2)$ by
\[
F(\omega_1,\omega_2) :=
\begin{cases}
F_{\omega_1} (\omega_2)
& \text{for a.e.}~ \omega_1 \in [0,a]
\;\; \text{and} \;\; \text{a.e.}~ \omega_2 \in [\frac{c}{a}\omega_1,\frac{c}{a}\omega_1 {+} s] , \\
0
& \text{otherwise} ,
\end{cases}
\]
which is clearly supported in $S$ and is square-integrable:
\[
\begin{split}
&\iint_S | F(\omega_1,\omega_2) |^2 \, d\omega_1 \, d\omega_2  \\
&= \int_{0}^{a} \int_{\frac{c}{a}\omega_1}^{\frac{c}{a}\omega_1 + s}
| F_{\omega_1} (\omega_2) |^2 \, d\omega_2 \, d\omega_1
= \int_{0}^{a} \| F_{\omega_1} \|_{L^2[\frac{c}{a}\omega_1,\frac{c}{a}\omega_1 + s]}^2 \, d\omega_1 \\
&\overset{\eqref{eqn:J-is-sampling-set-with-universal-constants}}{\leq}
\tfrac{1}{A'} \int_{0}^{a} \sum_{j \in J} |Q_j (\omega_1)|^2 \, d\omega_1
\overset{\eqref{eqn:int-sumJ-Hj-omega1-is-finite}}{\leq}
\tfrac{1}{A \, A'} \sum_{(x,j) \in \mathcal{C}} |b_{(x,j)}|^2
< \infty .
\end{split}
\]
The function $F$ essentially belongs in $L^2(S) \; (\hookrightarrow L^2(\R^2))$ and therefore $f := \mathcal{F}^{-1} [F]$ belongs in $PW (S)$.
Note that for any $j \in J$ and $x \in \mathcal{X}_j$,
\[
\begin{split}
f(x,j)
&= \iint F(\omega_1,\omega_2) \, e^{2 \pi i (x \omega_1 + j \omega_2)} \, d\omega_1 \, d\omega_2 \\
&= \int \underbrace{\Big[ \int F_{\omega_1} (\omega_2) \, e^{2 \pi i j \omega_2} \, d\omega_2 \Big]}_{= \mathcal{F}^{-1} [F_{\omega_1}] (j) = Q_j (\omega_1)} \, e^{2 \pi i x \omega_1} \, d\omega_1 \\
&= \int Q_j (\omega_1) \, e^{2 \pi i x \omega_1} \, d\omega_1
= \mathcal{F}^{-1} [Q_j] (x) = b_{(x,j)} .
\end{split}
\]
Hence, the function $f \in PW (S)$ satisfies $f(\gamma) = b_{\gamma}$ for all $\gamma \in \mathcal{C}$, as desired.
\end{proof}

\subsection{Proof of Proposition~\ref{prop:main2-lower-triangular} for $d \geq 2$}

Now, let us consider the general $d$-dimensional case. 
We may assume that 
\begin{equation}\label{eqn:Case1-parallel-y-axis-matrix-A-dxd}
H
=
\left( 
\begin{array}{c|c}
a \, & \, \mathbf{0} \\
\hline  \\[-4mm]
\mathbf{c} \, & \, H'
\end{array}
\right)
\end{equation}
with $0<a \leq 1$, $\mathbf{c}\in\mathbb{R}^{d-1}$, and $H' \in \R^{(d-1) {\times} (d-1)}$ a \emph{lower} triangular matrix with diagonals lying in $(0,1]$. 
Then
\[
H^{-1}
=
\left(
\begin{array}{c|c}
\frac{1}{a}  \, & \,  \mathbf{0} \\[1mm]
\hline  \\[-3mm]
\mathbf{z}  \, & \,  (H')^{-1}
\end{array}
\right) 
\quad \text{for some} \;\; \mathbf{z} \in \R^{d-1} , 
\]
so that 
\begin{equation}\label{eqn:A-T-form}
\quad
H^{-T}
=
\left(
\begin{array}{c|c}
\frac{1}{a} & \mathbf{z}^{*} \\[1mm]
\hline  \\[-3mm]
\mathbf{0} & (H')^{-T}
\end{array}
\right) 
\end{equation}
where $\mathbf{z}^{*} := (\overline{\mathbf{z}})^{T}$ denotes the conjugate transpose (or the Hermitian transpose) of $\mathbf{z}$. 
Therefore, we have
\[ 
H^{-T}\mathbb{Z}^d = \left\{m(\tfrac{1}{a},\mathbf{0}) + (\mathbf{z}^{*T}\mathbf{n},(H')^{-T}\mathbf{n})
\;\colon\;
 m\in\mathbb{Z}, \; \mathbf{n}\in\mathbb{Z}^{d-1} \right\}
\]
so that for each $\boldsymbol{y}\in (H')^{-T}\mathbb{Z}^{d-1}$ the set $\{x\colon (x,\boldsymbol{y})\in\Lambda^*\}$ is a shifted copy of the lattice $\frac{1}{a}\mathbb{Z}$.

To establish Proposition~\ref{prop:main2-lower-triangular} for $d \geq 2$, it is enough show the following.

\begin{proposition} 
Let $H$ be a  matrix of the form \eqref{eqn:Case1-parallel-y-axis-matrix-A-dxd}. 
Then there exists a set $\mathcal{C}\subseteq \Z^d$
such that $\mathcal{E}(\mathcal{C})$ is a Riesz basis for $L^2(H [0,1]^d)$.
Moreover, we may choose 
\begin{enumerate}[(i)]
\item
$\mathcal{C} = \mathbf{r} ( H^{-T}\mathbb{Z}^{d} )$, or  

\item
$\mathcal{C} = \mathrm{r} \big(\frac{1}{a_{11}} \Z \big) \times \cdots \times \mathrm{r} \big(\frac{1}{a_{dd}} \Z\big)$, where $a_{11}, \ldots, a_{dd}$ are the diagonals of $H$. 
\end{enumerate}  
\end{proposition}

\begin{remark}
\rm
The set $\mathrm{r} \big(\frac{1}{a_{11}} \Z \big) \times \cdots \times \mathrm{r} \big(\frac{1}{a_{dd}} \Z\big)$ in ({\romannumeral 2}) can be easily replaced with $\mathrm{r} \big(\frac{1}{a_{11}} \Z + \delta_1 \big) \times \cdots \times \mathrm{r} \big(\frac{1}{a_{dd}} \Z + \delta_d \big)$ for arbitrary constants $\delta_1, \ldots, \delta_d \in \R$ (see Lemma~\ref{lem:a-irrational-Beatty-Fraenkel}). 
\end{remark}

\begin{proof}
The proof proceeds by induction on $d$. The base case $d=2$ was proved in Section~\ref{subsec:case-d2} (see Proposition~\ref{prop:base-case-d2}).
Assume that the theorem holds for $d-1$ and show that it also holds for $d$.

By the induction hypothesis, we have that $\mathcal{E}(\mathcal{C}_{d-1})$ is a Riesz basis for $L^2(H'  [0,1]^{d-1})$, where the set $\mathcal{C}_{d-1} \subseteq \Z^{d-1}$ is either $\mathbf{r} ( (H')^{-T}\mathbb{Z}^{d-1} )$ or $\mathrm{r} \big(\frac{1}{\ell_{11}} \Z \big) \times \cdots \times \mathrm{r} \big(\frac{1}{\ell_{(d-1)(d-1)}} \Z\big)$ with $\ell_{11}, \ldots, \ell_{(d-1)(d-1)}$ denoting the diagonals of $H'$.

\begin{enumerate}[(i)]
\item
If $\mathcal{C}_{d-1} = \mathbf{r} ( (H')^{-T}\mathbb{Z}^{d-1} )$, then define $\mathcal{C}_{d} := \mathbf{r} ( H^{-T}\mathbb{Z}^{d} )$. Using \eqref{eqn:A-T-form}, we may also write
\begin{equation}\label{eqn:GeneralCase-rounding}
\mathcal{C}_{d} 
= \bigcup_{\boldsymbol{j} \,\in\, \mathcal{C}_{d-1}} \, \mathcal{X}_{\boldsymbol{j}} \times  \{ \boldsymbol{j}\} ,
\end{equation}
where for each $\boldsymbol{j}\in \mathcal{C}_{d-1}$, we define $\mathcal{X}_{\boldsymbol{j}} := \mathrm{r} (\{x\colon (x,\boldsymbol{y})\in H^{-T}\mathbb{Z}^{d} \})$ with $\boldsymbol{y}\in (H')^{-T}\mathbb{Z}^{d-1}$ being the closest element to $\boldsymbol{j}$.

\item
If $\mathcal{C}_{d-1} = \mathrm{r} \big(\frac{1}{\ell_{11}} \Z \big) \times \cdots \times \mathrm{r} \big(\frac{1}{\ell_{(d-1)(d-1)}} \Z\big)$, then define 
\[
\mathcal{C}_d 
:= \mathrm{r} \big(\tfrac{1}{a} \Z \big) \times \mathcal{C}_{d-1}   =  \mathrm{r} \big(\tfrac{1}{a} \Z \big) \times \mathrm{r} \big(\tfrac{1}{\ell_{11}} \Z \big) \times \cdots \times \mathrm{r} \big(\tfrac{1}{\ell_{(d-1)(d-1)}} \Z\big) , 
\]
where $a$, $\ell_{11}$, $\cdots$, $\ell_{(d-1)(d-1)}$ are the diagonals of $H$. 
We may also write  
\begin{equation}\label{eqn:GeneralCase-tensor}
\mathcal{C}_{d} 
= \bigcup_{\boldsymbol{j} \,\in\, \mathcal{C}_{d-1}} \, \mathcal{X}_{\boldsymbol{j}} \times  \{ \boldsymbol{j}\}  
\quad \text{with} \;\; \mathcal{X}_{\boldsymbol{j}} = \mathrm{r} \big(\tfrac{1}{a} \Z \big) \;\; \text{for all} \;\; \boldsymbol{j}\in \mathcal{C}_{d-1}.
\end{equation} 
\end{enumerate}   
Below, we will use the expressions \eqref{eqn:GeneralCase-rounding} and \eqref{eqn:GeneralCase-tensor} for $\mathcal{C}_{d}$. 
Using \eqref{eqn:Case1-parallel-y-axis-matrix-A-dxd}, we may write
\[
S := H [0,1]^d = \left\{ (\omega_1, \boldsymbol{\omega}_2) \;:\; 0 \leq \omega_1 \leq a , \;\;  \boldsymbol{\omega}_2\in \tfrac{\omega_1}{a} \mathbf{c}+H'[0,1]^{d-1} \right\} . 
\]


\noindent
\textbf{Completeness}. \
To show that $\mathcal{E}(\mathcal{C}_d)$ is complete in $L^2(S)$, we suppose that $f \in PW(S)$ satisfies $f(\gamma) = 0$ for all $\gamma \in \mathcal{C}_d$ and show that $f = 0$.

Note that since $f \in PW(S)$, its Fourier transform $F := \widehat{f} \in L^2(\R^d)$ is supported in $S \subseteq \R^d$. 
The set $S$ is bounded, so the function $f$ is continuous
and the following equation holds for all  
$x\in\R, \; \boldsymbol{y}\in\R^{d-1}$: 
\[
\begin{split}
f(x,y)
&= \iint F(\omega_1,\boldsymbol{\omega}_2) \, e^{2 \pi i (x \omega_1 + \boldsymbol{y}\cdot \boldsymbol{\omega}_2)} \, d\omega_1 \, d\boldsymbol{\omega}_2 \\
&= \int_{\R} \underbrace{\left( \int_{\R^{d-1}} F(\omega_1,\boldsymbol{\omega}_2) \, e^{2 \pi i \boldsymbol{y}\cdot \boldsymbol{\omega}_2)} \, d\boldsymbol{\omega}_2 \right) }_{=: P_{\boldsymbol{y}} F (\omega_1)} \, e^{2 \pi i x \omega_1} \, d\omega_1 \\
&= \mathcal{F}^{-1} [P_{\boldsymbol{y}} F] (x) ,
\end{split}
\]
where $\mathcal{F}$ denotes the $1$-dimensional Fourier transform and
$F \mapsto P_{\boldsymbol{y}} F$ is the linear mapping defined by $P_{\boldsymbol{y}} F (\omega_1) := \int_{\R^{d-1}} F(\omega_1,\boldsymbol{\omega}_2) \, e^{2 \pi i \boldsymbol{y}\cdot \boldsymbol{\omega}_2} \, d\boldsymbol{\omega}_2$.
Since $\supp F \subseteq S \subseteq [0,a] \times \R^{d-1}$, we have $\supp P_{\boldsymbol{y}} F \subseteq [0,a]$, so that $\mathcal{F}^{-1} [P_{\boldsymbol{y}} F] \in PW[0,a]$.
Recall that for each $\boldsymbol{j} \in \mathcal{C}_{d-1}$, the set $\mathcal{X}_{\boldsymbol{j}}$ is a set of uniqueness for $PW[0,a]$.
Since $f(\cdot,\boldsymbol{j}) = \mathcal{F}^{-1} [P_{\boldsymbol{j}} F] \in PW[0,a]$ and since $f(x,\boldsymbol{j}) = 0$ for all $x \in \mathcal{X}_{\boldsymbol{j}}$, we deduce that $f(\cdot,\boldsymbol{j}) = 0$, equivalently, $P_{\boldsymbol{j}} F = 0$.
Therefore, we have $P_{\boldsymbol{j}} F = 0$ for all $\boldsymbol{j} \in \mathcal{C}_{d-1}$.

Now, it holds for all $\boldsymbol{j} \in \mathcal{C}_{d-1}$ and $\omega_1 \in \R$ that
\[
0 = P_{\boldsymbol{j}} F (\omega_1)
:= \int_{\R^{d-1}} F(\omega_1,\boldsymbol{\omega}_2) \, e^{2 \pi i \boldsymbol{j}\cdot\boldsymbol{\omega}_2} \, d\boldsymbol{\omega}_2
= \mathcal{F}^{-1}_{d-1} [F(\omega_1,\cdot)] (\boldsymbol{j}) ,
\]
where $\mathcal{F}_{d-1}$ denotes the $(d{-}1)$-dimensional Fourier transform. 
Fix any $\omega_1 \in [0,a]$. Then
$\supp F(\omega_1,\cdot) \subseteq \frac{\omega_1}{a} \mathbf{c}+H'[0,1]^{d-1}$ by the geometry of $S$, which implies that $\mathcal{F}^{-1}_{d-1} [F(\omega_1,\cdot)] \in PW(\frac{\omega_1}{a} \mathbf{c}+H'[0,1]^{d-1})$.
Since $\mathcal{C}_{d-1} \subseteq \Z^{d-1}$ is a set of uniqueness for $PW(I)$ for any shift of $H'[0,1]^{d-1}$, we deduce that $\mathcal{F}^{-1}_{d-1} [F(\omega_1,\cdot)] = 0$ or equivalently, $F(\omega_1,\cdot) = 0$.
That is, $F(\omega_1,\cdot) = 0$ for all $\omega_1 \in [0,a]$.
Since $\supp F \subseteq S \subseteq [0,a] \times \R^{d-1}$, we conclude that $F = 0$ and in turn, $f = 0$.

\medskip

\noindent
\textbf{Riesz sequence}. \
One can easily see that the `Completeness' part above is only a minor modification of the `Completeness' part in the proof of Proposition~\ref{prop:base-case-d2}. The `Riesz sequence' part follows by a similar modification of the corresponding part in the proof of Proposition~\ref{prop:base-case-d2}. 
\end{proof}

\section{Proof of Theorem \ref{thm:limiting-spectral-norm}}
\label{sec:proof-main1}

As seen in Proposition~\ref{prop:Fuglede-for-lattices}, for any $A \in \mathrm{GL} (d,\R)$ the system $\mathcal{E}(A^{-T} \, \Z^d)$ is an orthogonal basis for $L^2(A [0,1]^d)$.
To prove Theorem \ref{thm:limiting-spectral-norm}, we will first derive a perturbation theorem for this orthogonal basis by using Theorem \ref{thm:Bailey-Cor6-1}.

\begin{theorem}\label{thm:Bailey-Cor6-1-generalization}
Let $A \in \mathrm{GL} (d,\R)$ and $0 < L < \frac{\ln(2)}{\pi d}$.
For each $\bn \in \Z^d$, let
\begin{equation}\label{eqn:blambda-in-neighborhood-of-dual}
\blambda_{\bn}
\;\;\in\;
A^{-T} \, \bn
+
A^{-T} \, [-L,L]^d .
\end{equation}
Then $\mathcal{E}( \{ \blambda_{\bn} : \bn \in \Z^d \} )$ is a Riesz basis for $L^2(A [0,1]^d)$.
\end{theorem}

\begin{proof}
From \eqref{eqn:blambda-in-neighborhood-of-dual}, we have $\btau_{\bn} := A^T \, \blambda_{\bn} \in \bn + [-L,L]^d$  for $\bn \in \Z^d$.
Then $\mathcal{E}( \{ \btau_{\bn} : \bn \in \Z^d \} )$ is a Riesz basis for $L^2[0,1]^d$ by Theorem \ref{thm:Bailey-Cor6-1}.
Note that since $A \in \mathrm{GL} (d,\R)$, the map $L^2[0,1]^d \rightarrow L^2(A [0,1]^d)$, $F(x) \mapsto F(A^{-1} x)$ is a bijective bounded operator.
For each $\bn \in \Z^d$, the operator maps the function $e^{2 \pi i \btau_{\bn} \cdot (\cdot)}$ to $e^{2 \pi i \btau_{\bn} \cdot A^{-1} (\cdot)} = e^{2 \pi i \btau_{\bn}^T A^{-1} (\cdot)} = e^{2 \pi i (A^{-T} \btau_{\bn})^T (\cdot)} = e^{2 \pi i \blambda_{\bn}^T (\cdot)}  = e^{2 \pi i \blambda_{\bn} \cdot (\cdot)}$. Therefore, $\mathcal{E}( \{ \blambda_{\bn} : \bn \in \Z^d \} )$ is a Riesz basis for $L^2(A [0,1]^d)$.
\end{proof}

Clearly, Theorem \ref{thm:Bailey-Cor6-1} is a special case of Theorem \ref{thm:Bailey-Cor6-1-generalization} with $A = I_d$ (the $d {\times} d$ identity matrix).

As a second ingredient for proving Theorem \ref{thm:limiting-spectral-norm}, we need the following lemma.
For a $d {\times} d$ matrix $A \in \R^{d \times d}$, let 
$\|A\|_2 := \sup_{\mathbf{x} \in \R^d , \, \| \mathbf{x} \|_2 \leq 1}   \| A \mathbf{x} \|_2$
be the usual spectral norm of $A$.
For $r > 0$, let $B_{r} (0) := \{ \bomega \in \R^d : \| \bomega \|_2 < r \}$ be the open ball in $\R^d$ with radius $r$ centered at the origin, and let $\overline{B}_{r} (0) = \{ \bomega \in \R^d : \| \bomega \|_2 \leq r \}$ be the closure of $B_{r} (0)$.  

\begin{lemma}\label{lem:singular-value}
For any $A \in \R^{d \times d}$ and $r > 0$,
we have $\|A\|_2 \leq r$ if and only if $A^T \left[ \overline{B}_{1} (0) \right] \subseteq \overline{B}_{r} (0)$.
\end{lemma}

\begin{proof}
Since
\[
\|A\|_2 = \|A^T\|_2 = \sup_{\mathbf{x} \in \R^d , \; \| \mathbf{x} \|_2 \leq 1}   \| A^T \mathbf{x} \|_2 ,
\]
we have
$A^T \left[ \overline{B}_{1} (0) \right] \subseteq \overline{B}_{\|A\|_2} (0)$ and $A^T \left[ \overline{B}_{1} (0) \right] \not \subseteq \overline{B}_{s} (0)$ for all $0 < s < \|A\|_2$.

If $\|A\|_2 \leq r$, then $A^T \left[ \overline{B}_{1} (0) \right] \subseteq \overline{B}_{\|A\|_2} (0) \subseteq \overline{B}_{r} (0)$. Conversely, if $A^T \left[ \overline{B}_{1} (0) \right] \subseteq \overline{B}_{r} (0)$ for some $r > 0$, then necessarily $r \geq \|A\|_2$. 
\end{proof}

We are now ready to prove Theorem \ref{thm:limiting-spectral-norm}.

\begin{proof}[Proof of Theorem \ref{thm:limiting-spectral-norm}]
Choose any $r \in \R$ with $\|A\|_2 \leq r < \frac{2\ln(2)}{\pi d^{3/2}}$. Then Lemma \ref{lem:singular-value} shows that
\[
A^T \left[ \overline{B}_{1} (0) \right]
\;\subset\;  \overline{B}_{r} (0)
\;\subset\;  \left[-r, r \right]^d
\]
and since the map $\mathbf{x} \mapsto A^T \mathbf{x}$ is linear we get
\[
A^T \left[ \overline{B}_{\sqrt{d}/2} (0) \right]
=
A^T \left[ \tfrac{\sqrt{d}}{2} \, \overline{B}_{1} (0) \right]
\;\subset\; \tfrac{\sqrt{d}}{2} \, \overline{B}_{r} (0)
\;\subset\;  \left[-\tfrac{r\sqrt{d}}{2}, \tfrac{r\sqrt{d}}{2} \right]^d .
\]
By setting $L := \frac{r\sqrt{d}}{2}  \; ( < \frac{\ln(2)}{\pi d} )$, we have
\[
A^T \left[ \overline{B}_{\sqrt{d}/2} (0) \right] \subseteq [-L,L]^d
\]
or equivalently,
\begin{equation}
\label{eqn:Bailey-Cor6-1-generalization-main-condition}
\overline{B}_{\sqrt{d}/2} (0) \subseteq A^{-T} \, [-L,L]^d .
\end{equation}
Note that the distance of an arbitrary point $\mathbf{x} \in \R^d$ to the integer lattice $\Z^d$ is at most $\sqrt{d}/2$.
If an element $A^{-T} \bn \in \Lambda^*$ with $\bn \in \Z^d$ is rounded to its closest point $\blambda_{\bn} \in \Z^d$, then necessarily $\| \blambda_{\bn} - A^{-T} \bn \|_2 \leq \sqrt{d}/2$, that is, $\blambda_{\bn} \in A^{-T} \bn + \overline{B}_{\sqrt{d}/2} (0)$.
Combining with \eqref{eqn:Bailey-Cor6-1-generalization-main-condition}, this shows that \eqref{eqn:blambda-in-neighborhood-of-dual} holds for all $\bn \in \Z^d$.
Hence, $\mathcal{E}(\mathcal{C}) = \mathcal{E}( \{ \blambda_{\bn} : \bn \in \Z^d \} )$ is a Riesz basis for $L^2(A [0,1]^d)$ by Theorem \ref{thm:Bailey-Cor6-1-generalization}.
\end{proof}

\backmatter

%
%
%

\bmhead{Acknowledgments}

D.G.~Lee is supported by the National Research Foundation of Korea (NRF) grant funded by the Korean government (MSIT) (RS-2023-00275360).
G.E.~Pfander acknowledges support by the German Research Foundation (DFG) Grant PF 450/11-1.

%

\section*{Declarations}

\subsection*{Funding}
This work was supported by the National Research Foundation of Korea (NRF) grant funded by the Korean government (MSIT) (RS-2023-00275360), and the German Research Foundation (DFG) Grant PF 450/11-1.

\subsection*{Conflict of interest}

On behalf of all authors, the corresponding author states that there is no conflict of interest.

\subsection*{Availability of data and materials}

Not applicable. 

\subsection*{Authors' contributions}
All authors contributed equally to this work. All authors are first authors, and the authors' names are listed in alphabetical order.

%
%
%
%
%
%
%

\begin{appendices}

\section{Proof of Theorem~\ref{thm:tensor-product-Kadec-Avdonin}}
\label{sec:proof-thm:tensor-product-Kadec-Avdonin}

The proof is based on the observation that $e^{2 \pi i \btau \cdot \bomega} = e^{2 \pi i \tau^{(1)} \omega^{(1)}} \cdot \ldots \cdot e^{2 \pi i \tau^{(d)} \omega^{(d)}}$ for $\btau = (\tau^{(1)}, \ldots, \tau^{(d)}) , \, \bomega = (\omega^{(1)}, \ldots, \omega^{(d)}) \in \R^d$.
For simplicity, we will only consider the case $d=2$ with both sets $\Gamma^{(1)}$ and $\Gamma^{(2)}$ satisfying \eqref{eqn:Kadec-condition}.

By Proposition~\ref{prop:Kadec-1-4} (Kadec's $\frac{1}{4}$-theorem), the system $\mathcal{E}(\Gamma^{(1)})$ (resp.~$\mathcal{E}(\Gamma^{(2)})$) is a Riesz basis for $L^2[0,1]$ with bounds $(1 - B(L^{(1)}))^2$ and $(1 + B(L^{(1)}))^2$ (resp.~with bounds $(1 - B(L^{(2)}))^2$ and $(1 + B(L^{(2)}))^2$).
For any finitely supported complex-valued sequence $\{ c_{(m, n)} : (m, n) \in \Z^2 \}$, we have
\begin{equation}\label{eqn:SZ99-RS-inequality-derivation1}
\begin{split}
& \Big\| \sum_{(m,n) \in \Z^2} c_{(m, n)} \, e^{2 \pi i (\tau_{m}^{(1)},\tau_{n}^{(2)}) \cdot (\cdot)} \Big\|_{L^2[0,1]^2}^2 \\
&=
\int_0^1 \int_0^1  \Big| \sum_{m \in \Z} \sum_{n \in \Z} c_{(m, n)} \, e^{2 \pi i (\tau_{m}^{(1)} \omega^{(1)} + \tau_{n}^{(2)} \omega^{(2)})} \Big|^2 \, d\omega^{(1)} \, d\omega^{(2)} \\
&=
\int_0^1 \left( \int_0^1  \Big| \sum_{m \in \Z} \Big( \sum_{n \in \Z} c_{(m, n)} \, e^{2 \pi i \tau_{n}^{(2)} \omega^{(2)}} \Big) \, e^{2 \pi i \tau_{m}^{(1)} \omega^{(1)}} \Big|^2 \, d\omega^{(1)} \right) \, d\omega^{(2)} \\
&\geq
\int_0^1 \left( (1 - B(L^{(1)}))^2 \, \sum_{m \in \Z} \Big| \sum_{n \in \Z} c_{(m, n)} \, e^{2 \pi i \tau_{n}^{(2)} \omega^{(2)}} \Big|^2 \right) \, d\omega^{(2)} \\
&\qquad =
(1 - B(L^{(1)}))^2 \, \sum_{m \in \Z} \int_0^1 \Big| \sum_{n \in \Z} c_{(m, n)} \, e^{2 \pi i \tau_{n}^{(2)} \omega^{(2)}} \Big|^2 \, d\omega^{(2)} ,
\end{split}
\end{equation}
where we used that for any fixed $\omega^{(2)} \in [0,1]$, the sequence $\{ \sum_{n \in \Z} c_{(m, n)} \, e^{2 \pi i \tau_{n}^{(2)} \omega^{(2)}} : m \in \Z \}$ is finitely supported, and that $\mathcal{E}(\Gamma^{(1)}) = \mathcal{E}(\{ \tau_m^{(1)} \}_{m \in \Z})$ is a Riesz basis for $L^2[0,1]$ with lower bound $(1 - B(L^{(1)}))^2$.
On the other hand, using that $\mathcal{E}(\Gamma^{(2)}) = \mathcal{E}(\{ \tau_n^{(2)} \}_{n \in \Z})$ is a Riesz basis for $L^2[0,1]$ with lower bound $(1 - B(L^{(2)}))^2$, we have for each $m \in \Z$,
\begin{equation}\label{eqn:SZ99-RS-inequality-derivation2}
\int_0^1 \Big| \sum_{n \in \Z} c_{(m, n)} \, e^{2 \pi i \tau_{n}^{(2)} \omega^{(2)}} \Big|^2 \, d\omega^{(2)}
\geq
(1 - B(L^{(2)}))^2 \, \sum_{n \in \Z} |c_{(m, n)}|^2 .
\end{equation}
Thus, combining \eqref{eqn:SZ99-RS-inequality-derivation1} and \eqref{eqn:SZ99-RS-inequality-derivation2} gives
\[
\begin{split}
& \Big\| \sum_{(m,n) \in \Z^2} c_{(m, n)} \, e^{2 \pi i (\tau_{m}^{(1)},\tau_{n}^{(2)}) \cdot (\cdot)} \Big\|_{L^2[0,1]^2}^2  \\
&\qquad\qquad \;\geq\;
(1 - B(L^{(1)}))^2 \, (1 - B(L^{(2)}))^2 \, \sum_{m \in \Z} \sum_{n \in \Z} |c_{(m, n)}|^2  ,
\end{split}
\]
establishing the lower Riesz inequality of $\mathcal{E}( \Gamma^{(1)} {\times} \Gamma^{(2)})$ for $L^2[0,1]^2$. The upper Riesz inequality is obtained similarly:
\[
\begin{split}
& \Big\| \sum_{(m,n) \in \Z^2} c_{(m, n)} \, e^{2 \pi i (\tau_{m}^{(1)},\tau_{n}^{(2)}) \cdot (\cdot)} \Big\|_{L^2[0,1]^2}^2  \\
&\qquad\qquad \;\leq\;
(1 + B(L^{(1)}))^2 \, (1 + B(L^{(2)}))^2 \, \sum_{m \in \Z} \sum_{n \in \Z} |c_{(m, n)}|^2  .
\end{split}
\]
Note that since $\mathcal{E}(\Gamma^{(1)})$ and $\mathcal{E}(\Gamma^{(2)})$ are complete in $L^2[0,1]$, their tensor product $\mathcal{E}( \Gamma^{(1)} {\times} \Gamma^{(2)}) = \mathcal{E}(\Gamma^{(1)}) \otimes \mathcal{E}(\Gamma^{(2)})$ is complete in $L^2[0,1]^2$.
Hence, we conclude that $\mathcal{E}( \Gamma^{(1)} {\times} \Gamma^{(2)})$ is a Riesz basis for $L^2[0,1]^2$ with bounds $(1 - B(L^{(1)}))^2 \, (1 - B(L^{(2)}))^2$ and $(1 + B(L^{(1)}))^2 \, (1 + B(L^{(2)}))^2$ (see e.g., \cite[p.27, Theorem 9]{Yo01}).

The same arguments can be used to prove the case where some sets $\Gamma^{(k)}$ satisfy \eqref{eqn:Avdonin-condition} instead of \eqref{eqn:Kadec-condition}; the only difference in this case is that explicit Riesz bounds cannot be obtained. 
\hfill $\Box$ 

\section{Proof of Lemma~\ref{lem:RB-basic-operations}}
\label{sec:proof-lem:RB-basic-operations}

To prove ({\romannumeral 1}), fix any $a \in \R^d$ and consider the translation operator $F(x) \mapsto F(x+a)$ which we denote by $T_{-a}$. 
This is a unitary operator from $L^2(S)$ onto $L^2(S + a)$, and we have 
$T_{-a} [ \mathcal{E}(\Lambda) ] = \{ e^{2\pi i \lambda \cdot (x+a)} : \lambda \in \Lambda \}
= \{ e^{2\pi i \lambda \cdot a} \, e^{2\pi i \lambda \cdot x} : \lambda \in \Lambda \}$. 
Since $e^{2\pi i \lambda \cdot a} \in \C$ are just some phase factors lying on the unit circle in $\C$, it follows that $\mathcal{E}(\Lambda)$ is a Riesz basis for $L^2(S + a)$ with bounds $\alpha$ and $\beta$.

Part ({\romannumeral 2}) follows similarly by noting that the modulation $F(x) \mapsto e^{2 \pi i b \cdot x} F(x)$ is a unitary operator on $L^2(S)$.

To prove ({\romannumeral 3}), fix any $A \in \mathrm{GL} (d,\R)$ and assume that 
$\mathcal{E}(\Gamma)$ is a Riesz basis for $L^2(S)$ with bounds $0 < \alpha \leq \beta < \infty$, where $\Gamma \subseteq \R^d$ is a discrete set and $S \subseteq \R^d$ is a measurable set. 
Then for any $\{ c_\gamma \}_{\gamma \in \Gamma} \in \ell_2 (\Gamma)$, we have
\[
\begin{split}
\Big\| \sum_{\bgamma \in \Gamma} c_{\bgamma} \, e^{2 \pi i A \bgamma \cdot \bx} \Big\|_{L^2(A^{-T} S)}^2 
&= \int_{A^{-T} S}  \, \Big| \sum_{\bgamma \in \Gamma} c_{\bgamma} \, e^{2 \pi i A \bgamma \cdot \bx} \Big|^2 \, d \bx \\
&= \int_{A^{-T} S}  \, \Big| \sum_{\bgamma \in \Gamma} c_{\bgamma} \, e^{2 \pi i \bgamma \cdot A^T \bx} \Big|^2 \, d \bx \\
&=  \int_{S}  \, \Big| \sum_{\bgamma \in \Gamma} c_{\bgamma} \, e^{2 \pi i \bgamma \cdot \bz} \Big|^2 \, | \! \det (A^{-T}) | \, d \bz  \quad \big( \bx = A^{-T} \bz \big) \\
&=  \frac{1}{| \! \det (A) |}  \, \int_{S}  \, \Big| \sum_{\bgamma \in \Gamma} c_{\bgamma} \, e^{2 \pi i \bgamma \cdot \bz} \Big|^2 \, d \bz   \\
&=  \frac{1}{| \! \det (A) |}  \, 
\Big\| \sum_{\bgamma \in \Gamma} c_{\bgamma} \, e^{2 \pi i \bgamma \cdot \bz}\Big\|_{L^2(S)}^2 
\end{split}
\]
where 
\[
\alpha \, \sum_{\gamma \in \Gamma} | c_\gamma |^2
\;\leq\;
\Big\| \sum_{\bgamma \in \Gamma} c_{\bgamma} \, e^{2 \pi i \bgamma \cdot \bz}\Big\|_{L^2(S)}^2
\;\leq\;
\beta \, \sum_{\gamma \in \Gamma} | c_\gamma |^2 . 
\]
This implies that $\mathcal{E}(A \Gamma)$ is a Riesz basis for $L^2(A^{-T} S)$ with bounds $\frac{\alpha}{| \! \det (A) |} $ and $\frac{\beta}{| \! \det (A) |} $.
In particular, if $A = c I_d$ for some $c > 0$, then $\mathcal{E}(c \Gamma)$ is a Riesz basis for $L^2(\frac{1}{c} S)$ with bounds $c^{-d} \alpha$ and $c^{-d} \beta$. 
\hfill $\Box$ 

\section{Proof of Lemma \ref{lem:a-irrational-Beatty-Fraenkel}}
\label{sec:proof-lem:a-irrational-Beatty-Fraenkel}

Lindner \cite{Li99} obtained an explicit lower Riesz bound for Avdonin's theorem (Proposition~\ref{prop:Avdonin-1-4-in-the-mean}).

\begin{proposition}[Theorem 1 in \cite{Li99}]
\label{prop:thm1-in-Li99}
Let $\{ \tau_n = n + \delta_n \}_{n \in \Z} \subseteq \R$ be a sequence with $\inf \{ |\tau_n - \tau_{n'}| : n,n'\in\Z ,  \; n \neq n' \} \geq \delta > 0$ and $\sup_{n \in \Z} | \delta_n | \leq B$.
Assume that there exists an integer $P \in \N$ satisfying the Avdonin $\frac{1}{4}$-condition \eqref{eqn:Avdonin-condition}, i.e.,
\[
\sup_{m \in \Z} \;
\tfrac{1}{P}
\Big| \sum_{k=mN}^{(m+1)P-1} \delta_k \Big|
\;\leq\; L
\;<\; \tfrac{1}{4} .
\]
Then $\mathcal{E}( \{ \tau_n \}_{n \in \Z} )$ is a Riesz basis for $L^2[0,1]$ with lower Riesz bound given by
\[
A(B,\delta,L,P) := \exp \big( - 20 \pi^2 (2\widetilde{B})^{2\widetilde{P}} / \widetilde{P}^2 \big) \cdot \big( \tfrac{\widetilde{\delta}}{9\widetilde{B}} \big)^{240 (2\widetilde{B})^{\widetilde{P}}} ,
\]
where
\[
\widetilde{P} := P \cdot \lceil \tfrac{1}{P} \cdot \tfrac{2(4B+2)^2}{1/4-L} \rceil,
\quad
\widetilde{B} := \tfrac{3}{2} + 2(3B+1) ,
\quad
\widetilde{\delta} := \tfrac{1}{2} \big( \tfrac{1}{4} - L \big) \delta .
\]
\end{proposition}

Using Lemma \ref{lem:RB-basic-operations},
we see that if $\alpha > 0$, $\beta \in \R$, and if $\{ \tau_n = \frac{n + \beta}{\alpha}  + \delta_n \}_{n \in \Z} \subseteq \R$ is a separated sequence with $\inf \{ |\tau_n - \tau_{n'}| : n,n'\in\Z ,  \; n \neq n' \} \geq 1$, $\sup_{n \in \Z} | \delta_n | \leq \frac{1}{2}$ and
\begin{equation}\label{eqn:Avdonin-condition-dilated-by-a}
\sup_{m \in \Z} \;
\tfrac{1}{P} \;
\Big| \sum_{k=mP}^{(m+1)P-1} \delta_k \Big|
\;\leq\; L
\;<\; \tfrac{1}{4\alpha}
\quad \text{for some} \;\; P \in \N,
\end{equation}
then $\mathcal{E}(\{ \tau_n \}_{n \in \Z}  )$ is a Riesz basis for $L^2[0,\alpha]$ with lower Riesz bound $A > 0$ \emph{depending only on $\alpha, L$ and $P$}.

\begin{proposition}[Theorem 26 in \cite{PRW21}]
\label{prop:thm26-in-PRW21}
For a piecewise continuous function $f : [0,1] \rightarrow \R$, an irrational number $\alpha \in \R$ and $\epsilon > 0$, there exists a number $P_0 = P_0 (f,\alpha,\epsilon) \in \N$ such that
\[
\Big|
\tfrac{1}{P} \sum_{k=mP}^{(m+1)P-1} f \Big( \tfrac{k+\beta}{\alpha} - \lfloor \tfrac{k+\beta}{\alpha} \rfloor \Big) - \int_0^1 f(t) \, dt
\Big|
< \epsilon
\]
for all integers $P \geq P_0$, $m \in \Z$ and $\beta \in \R$.
\end{proposition}

In particular, setting $f(t) = t$ yields the following corollary.

\begin{corollary}[Weyl--Khinchin equidistribution, see e.g., Theorem 17 in \cite{PRW21}]
\label{cor:thm17-in-PRW21}
For an irrational number $\alpha \in \R$ and $\epsilon > 0$, there exists a number $P_0 = P_0 (\alpha,\epsilon) \in \N$ such that
\begin{equation}\label{eqn:in-cor:thm17-in-PRW21}
\Big|
\tfrac{1}{P} \sum_{k=mP}^{(m+1)P-1} \tfrac{k+\beta}{\alpha} \; \mathrm{mod} \; 1 \;-\; \tfrac{1}{2}
\Big|
< \epsilon
\end{equation}
for all integers $P \geq P_0$, $m \in \Z$ and $\beta \in \R$.
\end{corollary}

We are now ready to prove Lemma \ref{lem:a-irrational-Beatty-Fraenkel}.

Let $0 < \alpha < 1$ be an irrational number and let $\epsilon = \frac{1}{8 \alpha}$.
By Corollary \ref{cor:thm17-in-PRW21}, we obtain a number $P_0 \in \N$ for which \eqref{eqn:in-cor:thm17-in-PRW21} holds, that is, for any integer $P \geq P_0$, $m \in \Z$ and $\beta \in \R$,
\[
\tfrac{1}{P} \;
 \Big|
\sum_{k=mP}^{(m+1)P-1} \Big( \tfrac{k+\beta}{\alpha} - \lfloor \tfrac{k+\beta}{\alpha} \rfloor -  \tfrac{1}{2} \Big)
\Big|
=
 \Big|
\tfrac{1}{P} \sum_{k=mP}^{(m+1)P-1} \tfrac{k+\beta}{\alpha} \; \mathrm{mod} \; 1 \;-\; \tfrac{1}{2}
\Big|
\;<\;   \tfrac{1}{8\alpha} .
\]
Fix any integer $P \geq P_0$.
Then
\[
\sup_{\beta \in \R} \;
\sup_{m \in \Z} \;
\tfrac{1}{P} \;
 \Big|
\sum_{k=mP}^{(m+1)P-1}
\Big(  \lfloor \tfrac{k+\beta}{\alpha} \rfloor + \tfrac{1}{2} - \tfrac{k+\beta}{\alpha} \Big)
\Big|
\;\leq\; \tfrac{1}{8\alpha}
\;<\; \tfrac{1}{4\alpha} .
\]
Note that since $\alpha < 1$, the set $\lfloor \frac{\Z+\beta}{\alpha} \rfloor {+} \frac{1}{2}  := \left\{ \lfloor \tfrac{k+\beta}{\alpha} \rfloor  {+} \frac{1}{2}  : k \in \Z \right\} \subseteq \Z {+} \frac{1}{2}$ has no repeated elements and every two elements are separated by distance at least $1$.
Note also that $| \lfloor \frac{k+\beta}{\alpha} \rfloor + \frac{1}{2}  - \frac{k+\beta}{\alpha} | \leq \frac{1}{2}$ for all $k \in \Z$.
Then we can deduce from \eqref{eqn:Avdonin-condition-dilated-by-a} with $L = \frac{1}{8\alpha}$ that $\mathcal{E}( \lfloor \frac{\Z+\beta}{\alpha} \rfloor {+} \frac{1}{2} )$ is a Riesz basis for $L^2[0,\alpha]$ with lower Riesz bound $A>0$ which depends only on $\alpha, P$ and is \emph{independent of $\beta \in \R$}.
Finally, using Lemma \ref{lem:RB-basic-operations}(b) we conclude that for every $\beta \in \R$ the system $\mathcal{E}( \lfloor \frac{\Z+\beta}{\alpha} \rfloor )$ is a Riesz basis for $L^2[0,\alpha]$ with lower Riesz bound $A>0$.

Now, let $0 < \alpha \leq 1$ be a rational number.
Since $\alpha < 1$, the set $\lfloor \frac{\Z+\beta}{\alpha} \rfloor  \subseteq \Z$ has no repeated elements for each $\beta \in \R$.
Moreover, each $\lfloor \frac{\Z+\beta}{\alpha} \rfloor  \subseteq \Z$ is a periodic set, which implies that there are only finitely many distinct sequences of the form $\lfloor \frac{\Z+\beta}{\alpha} \rfloor$, $\beta \in \R$. For each $\beta \in \R$, there exist a number $P \in \N$ and a constant $c \in \R$ such that
\begin{equation}\label{eqn:rational-periodic}
\sum_{k=mP}^{(m+1)P-1}
\Big( \lfloor \tfrac{k+\beta}{\alpha} \rfloor  - \tfrac{k+c}{\alpha}  \Big)
\;=\; 0
\quad \text{for all} \;\; m \in \Z .
\end{equation}
For instance, if $\alpha = \frac{2}{3}$,
there are the only three distinct sequences of the form $\lfloor \frac{3}{2} (\Z{+}\beta) \rfloor$, $\beta \in \R$:
\[
\begin{split}
& \{ \cdots , \; -6, \; -5, \; - 3, \; -2, \; 0, \; 1, \; 3, \; \cdots \}  \quad\quad \;\;\;\, (\text{obtained with} \; \beta = 0) ,   \\
& \{ \cdots , \; -6, \; -4, \; - 3, \; -1, \; 0, \; 2, \; 3, \; 5, \; 6, \; \cdots \}   \;\; (\text{obtained with} \; \beta = \tfrac{1}{3}) , \\
& \{ \cdots , \; -5, \; -4, \; - 2, \; -1, \; 1, \; 2, \; 4, \; 5, \; 7, \; \cdots \}  \;\; (\text{obtained with} \; \beta = \tfrac{2}{3}) ,
\end{split}
\]
all of which have period $3$.
In this case, for $\beta = 0$ one can choose $P= 2$ and $c = - \frac{1}{6}$; indeed,
\[
\begin{split}
\lfloor \tfrac{3}{2} (\Z{+}0) \rfloor &= \{ \cdots , \; - 3, \; -2, \; 0, \; 1, \; 3,  \; \cdots \}   ,   \\
\tfrac{3}{2} (\Z{-} \tfrac{1}{6} ) &= \{ \cdots , \; - 3.25, \; -1.75, \; -0.25, \; 1.25, \; 2.75,  \; \cdots \}
\end{split}
\]
shows that \eqref{eqn:rational-periodic} holds.
The condition \eqref{eqn:rational-periodic} means that $\lfloor \frac{\Z+\beta}{\alpha} \rfloor$ is a pointwise perturbation of a reference lattice $\frac{\Z+c}{\alpha}$ and in fact satisfies the Avdonin condition with $L = 0$ with respect to that lattice; hence, the system $\mathcal{E}( \lfloor \frac{\Z+\beta}{\alpha} \rfloor )$ is a Riesz basis for $L^2[0,\alpha]$ with optimal lower Riesz bound, say, $A_\beta >0$.
Since there are only finitely many distinct sequences of the form $\lfloor \frac{\Z+\beta}{\alpha} \rfloor$, $\beta \in \R$, we have $A := \min_{\beta \in \R} A_\beta > 0$.
Hence, we conclude that for every $\beta \in \R$, the system $\mathcal{E}( \lfloor \frac{\Z+\beta}{\alpha} \rfloor )$ is a Riesz basis for $L^2[0,\alpha]$ with lower Riesz bound $A > 0$.
\hfill $\Box$ 

%




\end{appendices}


\bibliography{submission-arXiv/parallelepipeds-bibliography}

\end{document}